\documentclass[10pt]{amsart}
\usepackage{amsmath,amsfonts,amssymb,amsthm}
\usepackage{graphics,graphicx}
\usepackage[colorlinks,hyperindex,linkcolor=blue]{hyperref}
\hypersetup{
pdfauthor = {A. del Pino Gomez, Aldo Witte},
 pdftitle= {Regularisation of Lie algebroids and Applications},
 pdfsubject = {Lie algebroids, foliations, b-geometry, contact structures}}
\usepackage{tikz-cd}

\newcommand{\A}{{\mathcal{A}}}
\newcommand{\U}{{\mathcal{U}}}
\newcommand{\Bcal}{{\mathcal{B}}}

\newcommand{\Dcal}{{\mathcal{D}}}

\newcommand{\Fcal}{{\mathcal{F}}}

\newcommand{\ff}{{\mathcal{F}}}
\newcommand{\Gcal}{{\mathcal{G}}}

\newcommand{\Lcal}{{\mathcal{L}}}

\newcommand{\Ucal}{{\mathcal{U}}}

\newcommand{\CC}{\mathbb{C}}
\newcommand{\DD}{\mathbb{D}}

\newcommand{\PP}{\mathbb{P}}

\newcommand{\RR}{\mathbb{R}}
\renewcommand{\SS}{\mathbb{S}}

\newcommand{\ZZ}{\mathbb{Z}}

\newcommand{\frakg}{\mathfrak{g}}

\newcommand{\End}{\operatorname{End}}

\newcommand{\Diff}{{\operatorname{Diff}}}

\newcommand{\hor}{\operatorname{hor}}

\newcommand{\can}{{\operatorname{can}}}

\newcommand{\NS}{\mathbb{S}}

\newcommand{\set}[1]{\lbrace #1\rbrace}
\newcommand{\inp}[1]{\langle #1\rangle}

\newcommand{\rest}[2]{\left.{#1}\right|_{#2}}
\newcommand{\rr}{\mathbb{R}}
\newcommand{\abs}[1]{\left| #1 \right|}
\newcommand{\reg}{\operatorname{reg}}
\newcommand{\pis}{\pi^{\sharp}}
\newcommand{\pd}[2]{\frac{\partial #1}{\partial #2}}
\DeclareMathOperator{\Res}{Res}

\newcommand{\elli}{\mathcal{A}_{\abs{D}}}
\newcommand{\cc}{\mathbb{C}}
\newcommand{\pbr}[1]{\{ #1\}}
\newcommand{\pbre}{\{ \cdot, \cdot \}}

\renewcommand{\Vert}{{\operatorname{Vert}}}

\usepackage{thmtools}
\declaretheorem[style=definition,qed=$\diamondsuit$]{definition}
\declaretheorem[style=definition,qed=$\triangle$,sibling=definition]{example}
\declaretheorem[style=plain,sibling=definition]{theorem}
\declaretheorem[style=plain,sibling=definition]{lemma}
\declaretheorem[style=plain,sibling=definition]{proposition}
\declaretheorem[style=plain,sibling=definition]{corollary}
\declaretheorem[style=definition,qed=$\diamondsuit$,sibling=example]{claim}
\declaretheorem[style=definition,qed=$\diamondsuit$,sibling=claim]{remark}
\numberwithin{theorem}{section}
\numberwithin{equation}{section}
\numberwithin{definition}{section}
\numberwithin{theorem}{section}
\numberwithin{proposition}{section}
\numberwithin{lemma}{section}
\numberwithin{example}{section}
\numberwithin{remark}{section}
\numberwithin{corollary}{section}

\numberwithin{equation}{section}
\newtheoremstyle{named}{}{}{\itshape}{}{\bfseries}{.}{.5em}{\thmnote{#3}#1}
\theoremstyle{named}
\newtheorem*{namedtheorem}{}

\title{Regularisation of Lie algebroids and applications}
\author{\'Alvaro del Pino and Aldo Witte}

\begin{document}
\begin{abstract}
We describe a procedure, called regularisation, that allows us to study geometric structures on Lie algebroids via foliated geometric structures on a manifold of higher dimension. This procedure applies to various classes of Lie algebroids; namely, those whose singularities are of $b^k$, complex-log, or elliptic type, possibly with self-crossings.

One of our main applications is a proof of the Weinstein conjecture for overtwisted $b^k$-contact structures. This was proven in \cite{MO21} using a certain technical hypothesis. Our approach avoids this assumption by reducing the proof to the foliated setting. As a by-product, we also prove the Weinstein conjecture for other Lie algebroids.

Along the way we also introduce tangent distributions, i.e. subbundles of Lie algebroids, as interesting objects of study and present a number of constructions for them.
\end{abstract}
\maketitle

%To Add? \alert{none of these would be necessary (and I hardly remember what the ideas where) but perhaps we have time for one or two (although I doubt that ;p)}.
%\begin{itemize}
%\item Classification of $\rr^+$-invariant contact structures
%\item $h$-principles?
%\item Holomorphic curves?
%\item Find examples of contact manifolds which are tight at infinity.
%\item Your examples with Fran with all orbits in the unique compact leaf through regularisation \alert{could that be?}
%\end{itemize}

\setcounter{tocdepth}{1}
\tableofcontents

%%%%%%%%%%%%%%%%%%%%%%%%%%%%%%%%%%%%%%%%%%%%%%%%%%%%%%%%%%
\section{Introduction}\label{sec:introduction}

Lie algebroids are the infinitesimal counterparts to Lie groupoids. They appeared first \cite{Prad} in the context of foliation theory and in the study of symmetry. Later on, they entered the world of geometric analysis in order to model various flavours of singularities. A classic example in this direction is the $b$-tangent bundle $\A_Z$ of Melrose \cite{M93}. It consists of those vector fields that are tangent to a given embedded hypersurface $Z$. We can think of $Z$ as a hypersurface at infinity, which is why $\A_Z$ plays a role in the study of the Atiyah-Patodi-Singer index theorem.

Some of the study of Lie algebroids has focused on particularly well-behaved families thereof. This is certainly the case for foliations (which are the Lie algebroids in which the anchor is a monomorphism) and $b$-tangent bundles (in which the anchor is an isomorphism almost everywhere, except along the hypersurface $Z$, where it vanishes transversely). Indeed, it is convenient to focus on Lie algebroids whose anchor is controlled in some way. In this article, we focus on Lie algebroids whose singularities are of $b^k$ \cite{S13}, elliptic \cite{CG17}, or self-crossing \cite{GLPR17,CKW20} type.

Our main goal in this article is to relate the aforementioned Lie algebroids to the foliated setting. The idea behind this is that foliations, being non-singular, are somehow easier to handle.
\begin{theorem} \label{thm:regularisationMain}
Let $M$ be a manifold endowed with a $b^k$-tangent bundle $\A_Z^k$ with underlying coorientable hypersurface $Z$. Then, there exists a codimension-$1$ foliation $\Fcal$ in $M\times \SS^1$, such that the projection to $M$ induces a Lie algebroid submersion $\Fcal \rightarrow \A_Z^k$.
\end{theorem}
The pair $(M\times \SS^1,\Fcal)$ is called the \textbf{regularisation} of $\A_Z^k$. In Section \ref{sec:regularisation} we state and prove various results with this flavour. They apply to each of the families of Lie algebroids we mentioned above.

\subsection{Lifting geometric structures}

The construction presented in Theorem \ref{thm:regularisationMain} was discovered already, in the $b$-setting, by Osorno-Torres \cite{OT15}. Using the language of $b$-symplectic structures, he proved:
\begin{proposition}
Let $(M,Z)$ be a manifold together with an embedded hypersurface, and let $\omega \in \Omega^2(\A_Z)$ be a $b$-symplectic form. Then, $M\times \SS^1$ admits a codimension-one symplectic foliation.
\end{proposition}
A similar statement, in the context of generalised complex geometry, later appeared in \cite{CG17}.

Compared to these statements already available in the literature, Theorem \ref{thm:regularisationMain} operates in the setting of Lie algebroids and does not require any additional geometry (i.e. symplectic or generalised complex). Nonetheless, once the regularisation is in place, we can indeed apply it to the study of geometric structures on Lie algebroids. Do note that there is an extensive body of work, within Poisson Geometry, studying symplectic structures on Lie algebroids \cite{NT99,GMP14,Lan21}. The contact case has also attracted considerable attention \cite{MO18,MO21,CO22} recently.

We have the following lifting statement:
\begin{theorem} \label{thm:liftingMain}
Let $M$ be a manifold endowed with a $b^k$-tangent bundle $\A_Z^k$ with underlying coorientable hypersurface $Z$. Let $(M\times \SS^1,\Fcal)$ be its regularisation. Let $\Gcal$ be a $b^k$-contact or symplectic structure. Then, there is a foliated contact or symplectic structure $\tilde\Gcal$ in $\Fcal$, lifting $\Gcal$.
\end{theorem}
In fact, there is nothing special about the contact case. Given any tangent distribution (i.e. vector subbundle) on a Lie algebroid, it can be similarly lifted to the regularisation (Proposition \ref{prop:regularisationDistributions}). Furthermore, lifting commutes with taking the curvature of the distribution, so all the relevant properties about it are preserved upon lifting.

\subsection{Singular Weinstein conjecture}

We now focus on Lie algebroid contact structures. The Weinstein conjecture states that every Reeb vector field associated to a contact form on a compact manifold admits a periodic orbit. Hofer \cite{Hof93} proved the conjecture for contact three-manifolds $(M^3,\alpha)$ under the additional assumption that $\alpha$ is overtwisted or $\pi_2(M) \neq \emptyset$. Later, Taubes \cite{Taub} proved the conjecture for general three-dimensional manifolds by localizing the Seiberg-Witten equations along orbits of the Reeb vector field. This circle of ideas still applies in the Lie algebroid setting. This was pioneered by Miranda and Oms \cite{MO18,MO21}, who proved the Weinstein conjecture (under a technical assumption called \emph{$\rr_+$-invariance}) in the framework of $b^k$-geometry.

The regularisation allows us to study singular contact structures via foliated contact structures. Using the foliated Weinstein conjecture \cite{dPP18}, we deduce:
\begin{namedtheorem}[Theorem \ref{th:Weinnormal}]
Let $(M^{2n+1},Z,\alpha)$ be a $b^k$-contact manifold, and assume that there exists an overtwisted disk in $M\setminus Z$. Then either:
\begin{itemize}
\item There is a periodic Reeb orbit in $M \setminus Z$.
\item There is a non-constant Reeb orbit in $Z$.
\end{itemize}
\end{namedtheorem}
Do note that, in the $b^k$-setting, there may be Reeb orbits in $Z$ that are constant. Upon regularising, these correspond to Reeb orbits in $Z \times \SS^1$  tangent to the second factor. The theorem states that the orbits we obtain are not of this type.

Assuming $\rr_+$-invariance, \cite{MO21} found Reeb orbits in $M^3 \setminus Z$. This is a corollary of Theorem \ref{th:Weinnormal}:
\begin{namedtheorem}[Theorem \ref{th:Weininvariant}]
Let $(M^{2n+1},Z,\alpha)$ be a closed overtwisted $b^k$-contact manifold. Assume that there is a neighbourhood of $Z$ in which $\alpha$ is $\rr^+$-invariant. Then, there are contractible Reeb orbits in $M\setminus Z$.
\end{namedtheorem} 
We will also prove similar Weinstein-style statements for other types of singular contact structures.

\subsection{Organisation of the paper}

Some of the elementary theory of Lie algebroids is recalled in Section \ref{sec:Lie algebroids}. Along the way we introduce tangent distributions, which we use to discuss Ehresmann connections on Lie algebroids. We then describe the classes of Lie algebroids that we will work with (Section \ref{sec:examples}). The regularisation procedure for Lie algebroids is discussed in Section \ref{sec:regularisation}. We present several interpretations of it, including its relation to the canonical representation. All throughout we work with the classes of Lie algebroids introduced in Section \ref{sec:examples}. The section closes with Subsection \ref{ssec:regularisationDistributions}, where we explain the correspondence between distributions in a Lie algebroid and in its regularisation.

We then focus on applications to Lie algebroid distributions. We first provide an overview of the contact setting (Section \ref{sec:contact}), which will then allow us to study the Weinstein conjecture in Section \ref{sec:Weinstein}. In particular, we will prove Theorems \ref{th:Weinnormal} and \ref{th:Weininvariant}, as well as similar results for other singular contact structures.

Section \ref{sec:otherDistributions} is a brief introduction to bracket-generating distributions in Lie algebroids. We present various constructions and examples.

The article also includes an Appendix \ref{sec:Poisson} where we explain how the regularisation procedure fits within the Poisson and Jacobi frameworks.

\emph{Acknowledgements}: The authors would like to thank L. Accornero, C. Kirchoff-Lukat, M. Crainic, F. Presas, A.G. Tortorella, and L. Vitagliano for useful discussions.

%%%%%%%%%%%%%%%%%%%%%%%%%%%%%%%%%%%%%%%%%%%%%%%%%%%%%%%%%%
\section{Lie algebroids} \label{sec:Lie algebroids}

In this section we recall some of the basic theory of Lie algebroids. This will allow us to introduce distributions (i.e. subbundles) in the Lie algebroid setting (Subsection \ref{ssec:distributions}). As a particular example, we discuss Lie algebroid connections in the sense of Ehresmann (Subsection \ref{ssec:connections}).

\subsection{Lie algebroids}

In this article we are interested in geometric structures with mild singularities. There are many in-equivalent geometric setups that allow for singularities, but the one that we concern ourselves with reads as follows:
\begin{definition}
A \textbf{Lie algebroid} on a manifold $M$ is a triple consisting of:
\begin{itemize}
    \item a vector bundle $\A \rightarrow M$,
    \item a bracket $[\cdot,\cdot]$ on its space of sections $\Gamma(\A)$,
    \item and a vector bundle morphism $\rho : \A \rightarrow TM$, called the \textbf{anchor},
\end{itemize}
satisfying:
\begin{equation*}
[v_1,fv_2] = f[v_1,v_2] + \mathcal{L}_{\rho(v_1)}(f)v_2,
\end{equation*}
for all $f \in C^{\infty}(M)$ and $v_1,v_2 \in \Gamma(\A)$.
\end{definition}

\begin{remark}
It is customary to just write $\A \rightarrow M$ for the Lie algebroid and leave the rest of the data implicit.
\end{remark}

\subsubsection{Morphisms}

Given Lie algebroids $\A \to M$ and $\Bcal \to N$, a \textbf{Lie algebroid morphism} is a bundle morphism 
\[ \Phi: \A \to \Bcal \]
lifting a map $\phi: M \to N$ and commuting with anchors and brackets.

\begin{example}
The tangent bundle $TM$, together with the usual Lie bracket of vector fields, and the identity as anchor, is a Lie algebroid. The unique Lie algebroid morphism $\Phi: TM \to TN$ between tangent bundles lifting a map $\phi: M \to N$ is the differential $d\phi$.

Similarly, the tangent space of a foliation $\ff \subset TM$ defines a Lie algebroid with the inclusion as anchor. A morphism of Lie algebroids between two foliated manifolds is necessarily the leafwise differential of a foliated map. 
\end{example}

\subsubsection{Differential forms}

Many geometric constructions associated to the tangent bundle can easily be extended to general Lie algebroids. For instance, we can define $\A$-differential forms using duality:
\[ \Omega^{\bullet}(\A) := \Gamma(\wedge^{\bullet}\A^*) \]
and the corresponding differential 
\[ d_{\A} : \Omega^{\bullet}(\A) \rightarrow \Omega^{\bullet+1}(\A). \]
using the Koszul formula. Some of the geometric structures to be considered later on will be written in terms of such forms.

\subsection{Distributions} \label{ssec:distributions}

The main objects of interest in this paper are distributions on Lie algebroids:
\begin{definition}
Let $M$ be a smooth manifold and $\mathcal{A} \rightarrow M$ a Lie algebroid. An \textbf{$\A$-distribution} of rank $k$ is a $k$-dimensional subbundle $\xi \subset \mathcal{A}$.
\end{definition}
We henceforth write $\Gamma(\xi) \subset \Gamma(\A)$ for the sheaf of sections of $\xi$.

\subsubsection{Lie flag}

In order to measure the non-involutivity of $\xi$ we associate to it a sequence of $C^\infty(M)$-modules, called the \textbf{Lie flag}:
\[ \Gamma(\xi)^{(1)} := \Gamma(\xi) \subset \Gamma(\xi)^{(2)} \subset \cdots \subset \Gamma(\xi)^{(i)} \subset \cdots \]
\[ \Gamma(\xi)^{(i+1)} := \langle [v,w] \mid v \in \Gamma(\xi)^{(1)}, w \in \Gamma(\xi)^{(i)} \rangle_{C^\infty}, \] 
where the bracket is the bracket on sections of $\A$.

We say that $\xi$ is \textbf{bracket-generating} if $\Gamma(\xi)^{(i_0)} = \Gamma(\A)$ for some $i_0$. The smallest such $i_0$ is called the \textbf{step} of $\xi$. At the opposite end of the spectrum, we say that $\xi$ is \textbf{involutive} if $\Gamma(\xi)^i = \Gamma(\xi)$ for every $i$. This is equivalent to the fact that the bracket preserves $\xi$, meaning that $\xi$ is a Lie subalgebroid.

\begin{example}
If $\rho$ is identically zero, $\A$ is a bundle of Lie algebras. Then, a bracket-generating $\A$-distribution $\xi$ is a subbundle of $\A$ that, fibrewise, is a generating set of the corresponding algebra. Similarly, $\xi$ is involutive if it restricts fibrewise to a Lie subalgebra.
\end{example}

\subsubsection{Curvature}

For simplicity, we will henceforth assume that $\xi$ is \textbf{regular}, i.e. that there are $\A$-distributions $\xi_i$ such that $\Gamma(\xi_i) = \Gamma(\xi)^{(i)}$. Much like in the classic study of distributions, using the Leibniz rule for the bracket we deduce that:
\begin{proposition} \label{prop:curvature}
For $\xi$ regular and  each pair of positive integers $i,j$, the Lie bracket defines a well-defined bundle morphism, called the \textbf{curvature}:
\[ \Omega_{i,j}(\xi) : \xi_i \times \xi_j \longrightarrow \xi_{i+j}/\xi_{i+j-1} \]
which factors through:
\[  \xi_i/\xi_{i-1} \times \xi_j/\xi_{j-1} \longrightarrow \xi_{i+j}/\xi_{i+j-1}. \]
\end{proposition}

\subsection{Lie algebroid connections} \label{ssec:connections}

Much like in the classic case, Lie algebroid Ehresmann connections are concrete examples of Lie algebroid distributions. Connections will play an important role in the constructions of Section \ref{sec:regularisation}. Namely, they will allow us to ``desingularise'' certain Lie algebroids with mild singularities to Lie algebroids of foliation type. 

Let $\A \rightarrow M$ and $\Bcal \rightarrow N$ be Lie algebroids. Suppose that $\Phi: \Bcal \to \A$ is a Lie algebroid morphism that is a fibrewise epimorphism and lifts a surjective submersion $\phi: N \to M$. We can then write $\Vert(\phi) := \ker(d\phi)$ and similarly $\Vert(\Phi) := \ker(\Phi)$. Due to the commutativity of the square
\begin{center}
\begin{tikzcd}
\Bcal \ar[r,"\Phi"] \ar[d] & \A \ar[d]\\
TN \ar[r,"d\phi"] & TM,
\end{tikzcd}
\end{center}
it follows that $\Vert(\Phi)$ is mapped to $\Vert(\phi)$ by the anchor of $\Bcal$.

\begin{definition} \label{def:connection}
Let $\A \to M$, $\Bcal \to N$, and $\Phi$ as above. An \textbf{Ehresmann connection} for $\Phi$ is a subbundle $H \subset \Bcal$ transverse to $\Vert(\Phi)$ and of complementary rank. An Ehresmann connection $H$ is \textbf{flat} if its curvatures (as in Proposition \ref{prop:curvature}) vanish.
\end{definition}
This means that the subsheaf of sections $\Gamma(H) \subset \Gamma(\Bcal)$ is involutive, $H \to N$ is a Lie algebroid, and $\Phi|_H: H \to \A$ is Lie algebroid morphism that, by construction, is a fibrewise isomorphism.

\subsubsection{Pullback Lie algebroids}

Let us discuss a concrete case of Definition \ref{def:connection}. Our starting data is a Lie algebroid $\rho: \A \rightarrow TM$ and a map $\phi: N \rightarrow M$ which is transverse to $\rho$. Transversality allows us to consider the pullback bundle
\begin{equation*}\label{eq:pullalg}
\phi^!\A := \set{(v,w) \in \A \times TN \,\mid\, \rho(v) = d\phi(w)},
\end{equation*}
which fits into the commutative square
\begin{center}
\begin{tikzcd}
\phi^!\A \ar[r,"d_{\A}\phi"] \ar[d,"\phi^!\rho"] & \A \ar[d,"\rho"]\\
TN \ar[r,"d\phi"] & TM.
\end{tikzcd}
\end{center}
Requiring that $d_{\A}\phi$ is a Lie algebroid morphism uniquely defines a Lie algebroid structure on $\phi^!\A \to N$ with anchor $\phi^!\rho: \phi^!\A \to TN$. By construction, $d_{\A}\phi$ is a fibrewise epimorphism and fits into our previous discussion.

\subsubsection{The linear case}

In the setting of vector bundles one can particularise this discussion to linear connections:
\begin{definition} \label{def:connectionLinear}
Let $\A \rightarrow M$ be a Lie algebroid, and $\pi :E \rightarrow M$ be a vector bundle. A (linear) $\A$-\textbf{connection} on $E$ is a map
\begin{align*}
\nabla : \Gamma(\A) \times \Gamma(E) &\quad\longrightarrow\quad \Gamma(E) \\
(v,w)  &\quad\mapsto\quad \nabla_v(w),
\end{align*}
which is $C^{\infty}(M)$-linear in the first entry, and satisfies
\[ \nabla_v(fw) = f\nabla_v(w) + \mathcal{L}_{\rho(v)}(w), \]
for all $f \in C^{\infty}(M)$.
\end{definition}
As in the classic case, a linear connection (as in Definition \ref{def:connectionLinear}) yields an Ehresmann connection (as in Definition \ref{def:connection}). This is easiest described using an associated horizontal lift operator 
\begin{equation*}
\hor^{\nabla} : \Gamma(\mathcal{A}) \rightarrow \Gamma(\pi^!\mathcal{A}),
\end{equation*}
for which the Ehresmann connection is the image. Write $\hor^{\nabla}(\alpha) = (\alpha,X) \in \Gamma(\A)\times \Gamma(TE)$. As $X$ is characterised by its action on linear functions $s \in \Gamma(E^*)$, one defines $X(s) = \nabla^*_{\rho(\alpha)}(s)$. 

Furthermore, linear $\A$-connections can be described by their connection matrices. Namely, given a cover $\set{U_{\alpha} \,\mid\, \alpha \in I}$ of $M$ by open balls on which $E$ trivialises, the connection is given by the associated matrix-valued one-forms 
\begin{equation*}
A_{\alpha} \in \Omega^1(\rest{\A}{U_{\alpha}};\text{End}(\rr^r)),
\end{equation*}
where $r$ is the rank of $E$. In these local trivialisations the horizontal lift is given by
\begin{equation*}
\hor^{\nabla}(X) = X - a(A_{\alpha}(X)),
\end{equation*}
where $a : \text{End}(E) \rightarrow \mathfrak{X}^1(E)$ is the map given by $a(B)(s) = B^*(s)$.

\subsubsection{Representations}

The (linear) curvature of the $\A$-connection is an element of $\Omega^2(\A;\text{End}(E))$, whose vanishing is equivalent to the vanishing of the curvature in the Ehresmann sense. In this case, the $\A$-connection is flat and we say that it is a $\A$-\textbf{representation}.

The main Lie algebroid representation we will use is the following:
\begin{definition}[\cite{ELW96}] \label{def:canonicalRep}
Let $\A \rightarrow M$ be a Lie algebroid. The \textbf{canonical representation} of $\A$ on 
\[ Q := \det(A^*)\otimes \det(TM) \]
is defined by
\begin{equation*}
\nabla_v(\alpha \otimes \omega) = \Lcal_v(\alpha) \otimes \omega + \alpha \otimes \Lcal_{\rho(v)}(\omega). \hfill \qedhere
\end{equation*}
\end{definition}

%%%%%%%%%%%%%%%%%%%%%%%%%%%%%%%%%%%%%%%%%%%%%%%%%%%%%%%%%%%%%
\section{Some relevant families of Lie algebroids} \label{sec:examples}

In this section we discuss some concrete families of Lie algebroids, putting particular emphasis on Lie algebroids of divisor type. The material in this section is standard and we refer the reader to \cite{K18} for a more detailed introduction.

\subsection{$b$-Geometry} \label{ssec:logGeometry}

All the Lie algebroids to be considered in this paper are either of foliation type or have an anchor that is an isomorphism everywhere except along a submanifold (or a union of submanifolds with controlled intersections).

The simplest type of singularity one may encounter occurs along an embedded hypersurface. Fix a manifold $M$ and a hypersurface $Z$; we will say that the pair $(M,Z)$ is a \textbf{b-pair}. It can be checked that the sheaf of vector fields tangent to $Z$ is locally free. Consequently, the Serre-Swann theorem ensures the existence of a vector bundle $\A_Z \rightarrow M$ whose sections are precisely the vector fields tangent to $Z$. 

Taking local coordinates $(z,x_2,\ldots,x_n)$ in which $Z = \set{z=0}$, we have that vector fields tangent to $Z$ are given as the span
\begin{align*}
\inp{z\partial_z,\partial_{x_2},\ldots,\partial_{x_n}},
\end{align*}
and the vector fields shown are thus a frame for $\A_Z$.
\begin{definition}
Let $(M,Z)$ be a b-pair. The vector fields tangent to $Z$ are the sections of a Lie algebroid, $\A_Z$, called the \textbf{b-tangent bundle}.
\end{definition}
These objects show up naturally in various areas of mathematics. In the context of PDE's, their study is called $b$-geometry \cite{M93}. Here ``$b$'' stands for ``boundary'', and indeed $Z$ is thought of as the ideal boundary of $M \setminus Z$ at infinity. In the context of algebraic geometry, the study of these objects is called log(arithmic) geometry, which deals with divisors in algebraic varieties.

\subsection{$b^k$-Geometry} \label{ssec:bkGeometry}

Let $M$ be a manifold, $\iota: Z \hookrightarrow M$ an embedded hypersurface, and $k>1$. One may wish to consider $k$-order tangencies of vector fields with the hypersurface $Z$. As explained in \cite{S13}, this notion is not well-defined without the fixing of auxiliary data.

Consider the sheaf of $k$-jets at $Z$, defined as $J_Z^k = \iota^{-1}(C^{\infty}_M/I_Z^{k+1})$. Given a function $f$, defined on a neighbourhood of $Z$, we denote by $j^k_Zf \in J_Z^k$ its jet along $Z$.
\begin{definition}[\cite{S13}] \label{def:Scott}
Fix a jet $j \in J_Z^{k-1}$ with vanishing zero-jet and non-vanishing first-jet. I.e. it is representable by a function cutting $Z$ transversely as its zero locus.

The \textbf{$b^k$-tangent bundle} associated to $M$, $Z$ and $j$ is the Lie algebroid $\A_Z^k \rightarrow M$ whose sections are
\begin{equation*}
\mathfrak{X}(I_Z^k,j) :=\set{ X \in \mathfrak{X}^1(M) : \mathcal{L}_X(f) \in I_Z^k \text{ for all } f \in j}.\hfill \qedhere
\end{equation*}
\end{definition}
Henceforth, whenever we talk about $b^k$-tangent bundles, we will consider jets $j$ that satisfy the (non-)vanishing hypotheses of the definition. Such a jet $j$ defines a coorientation of $Z$. Indeed, we can represent $j$ locally by a function $f$ cutting $Z$ transversely. This implies that $Z$ is locally two-sided and the side in which $f$ is positive may be regarded as the positive one. However, do note that this coorientation is not canonically given by the algebroid:
\begin{lemma} \label{lem:scalingJ}
Let $j_0$ and $j_1$ be jets in $J_Z^{k-1}$. Then, the associated Lie algebroids agree if and only if $j_0 = \lambda j_1$, with $\lambda \in \RR^*$.
\end{lemma}

There are Lie algebroids that are locally modelled on Definition \ref{def:Scott} but that cannot be described globally in terms of a jet $j$, due to the presence of non-trivial holonomy. We refer the reader to the companion paper \cite{dPW}.

In order to provide a unified treatment, we will still speak of $b^k$-Geometry with $k=1$. In that case, the data of $j$ is vacuous.

\subsection{Ideals and divisors}

A useful tool in $b$-geometry (as well as in some of the other singular geometries to be introduced in the upcoming subsections) is the notion of a divisor.
\begin{lemma}[\cite{vdLD16}]\label{lem:divisor}
Let $I \subset C^{\infty}(M;\rr)$ be a locally principal ideal, locally generated by functions with nowhere dense zero-set. Then there exists a line bundle $W\rightarrow M$ and a section $s \in \Gamma(W)$ such that
\[ I_s := s(\Gamma(W^*)) \subset C^{\infty}(M) \]
coincides with $I$. This line bundle $W$ is unique up to isomorphism, and the section $s$ is unique up to scaling.
\end{lemma}
\begin{proof}
Let $\set{U_{\alpha} : \alpha \in I}$ be an open cover of $M$. For each $\alpha \in I$ let $s_{\alpha}$ be a local generator of $I$ on $U_{\alpha}$. On overlaps $U_{\alpha} \cap U_{\beta}$ there are functions $g^{\alpha}_{\beta} \in C^{\infty}(U_{\alpha} \cap U_{\beta})$ such that $s_{\alpha} = g^{\alpha}_{\beta}s_{\beta}$. These functions $g^{\alpha}_{\beta}$ are then the transition functions of the claimed vector bundle $W$. The section $s$ has $s_{\alpha}$ as its local expression on $U_{\alpha}$.
\end{proof}
The pair $(W,s)$, as introduced in the Lemma \ref{lem:divisor}, is said to be the \textbf{b-divisor} associated to $(M,Z)$. We now collect a short list of useful remarks/constructions:

\begin{remark}
Due to the transverse vanishing of $s$, it follows that the differential of $s$ along the normal bundle of $Z$
\begin{equation*}
d^{\nu}s : \nu(Z) \rightarrow \rest{W}{Z}
\end{equation*}
is an isomorphism.
\end{remark}

\begin{remark}
Let $(M,Z)$ be a b-pair. A defining function exists in a neighbourhood of $Z$ if and only if $Z$ is co-orientable. A global defining function for $Z$ if and only if the associated line bundle $W$ is trivial. 
\end{remark}

\begin{remark} \label{rem:logconn}
As shown in \cite{K17}, the b-divisor can be recovered from the Lie algebroid $\A_Z$ using the canonical representation (Definition \ref{def:canonicalRep}). Namely: $W \simeq \det(\A^*_Z)\otimes \det(TM)$, and under this isomorphism $s$ will be sent to $\det(\rho_{\A_Z})$. This shows that associated to any b-pair $(M,Z)$ there is a completely canonical b-divisor:
\begin{equation*}
(\det(\A^*_Z)\otimes \det(TM),\det(\rho_{\A_Z})). \hfill \qedhere
\end{equation*}
\end{remark}

\begin{definition}
Let $(M,Z)$ be a b-pair, and let $\alpha \in \Omega^k(\A_Z)$. Its \textbf{residue} is the form
\[ \Res(\alpha) \in \Omega^{k-1}(Z) \]
which in local coordinates in which $Z = \set{z=0}$ is given by $\iota^*_Z(\iota_{z\partial_z}\alpha)$.
\end{definition}

\subsection{Complex log geometry}

The notions introduced in the previous subsection have analogues in the complex setting. Namely: a \textbf{complex log divisor} is a pair $(L,\sigma)$ consisting of a complex line bundle $L \longrightarrow M$ and a section $\sigma: M \to L$ that vanishes transversely.
\begin{definition}
Let $(L,\sigma)$ be a complex log divisor, let 
\[ I_{\sigma} := \sigma(\Gamma(L^*)) \subset C^{\infty}(M;\cc) \]
be the ideal induced by $\sigma$, and let $D = \sigma^{-1}(\set{0})$.

Then, complex vector fields preserving $I_\sigma$ form a complex Lie algebroid, called the \textbf{complex log tangent bundle}, and denoted by $\A_D$.
\end{definition}
As the notation suggests, the resulting algebroid depends only on $D$ (and its coorientation induced by $\sigma$), and not in the concrete pair $(L,\sigma)$ producing it (see Corollary 1.18 in \cite{CG17}).

\begin{remark}\label{rem:complexlogconn}
As explained in Remark \ref{rem:logconn}, one can show that in fact $L \simeq \det(\A_D^*)\otimes \det(TM)$, and that under this isomorphism $\sigma$ is send to $\det(\rho_{\A_D})$. Consequently, there is a completely canonical complex log divisor associated to $\A_D$:
\begin{equation*}
(\det(\A_D^*)\otimes \det(TM),\det(\rho_{\A_D})). \hfill \qedhere
\end{equation*}
\end{remark}

\subsection{Elliptic geometry}

Suppose $(L,\sigma)$ is a complex log divisor. One can then take the ``absolute value squared'' of $(L,\sigma)$ by considering the real line bundle with section $(W,s)$ defined by 
\[ (L\otimes \bar{L},\sigma \otimes \bar{\sigma}) = (W,s) \otimes \cc. \]
This will provide the following kind of object:
\begin{definition}[\cite{CG17}]
An \textbf{elliptic divisor} $(W,s)$ is a real line bundle $W \longrightarrow M$ with a section $s: M \to W$ which, in every trivialisation, corresponds to a definite Morse--Bott function with codimension two critical set $D \subset M$.
\end{definition}
Conversely, if $(W,s)$ is an elliptic divisor with $s^{-1}(\set{0})$ co-orientable, then there exists a complex log divisor $(L,\sigma)$ such that $(L\otimes \bar{L},\sigma \otimes \bar{\sigma}) = (W,s) \otimes \cc$, unique up to isomorphism \cite[Prop. 2.26]{CG17,CKW20}.

At the level of Lie algebroids we obtain:
\begin{definition}[\cite{CG17}]
Let $(W,s)$ be an elliptic divisor, let $I_s := s(\Gamma(W^*))$ be the ideal induced by $s$, and let $D$ be the critical set of $s$.

Then, the real vector fields preserving $I_s$ form a Lie algebroid, called the \textbf{elliptic tangent bundle}, which we denote by $\A_{\abs{D}}$.
\end{definition}

\subsection{Self-crossing geometries}

Products are not defined in the $b$-category or in the elliptic category, since a product manifold has a divisor (the union of preimages of each of the divisors coming from the factors) that intersects itself. This leads us to the following definitions:
\begin{definition}\label{def:logx}
Let $(M,Z)$ be a manifold together with an immersed hypersurface $Z$, which is the union $\cup_i Z_i$ of finitely many transversely\footnote{A collection of submanifolds $\set{Z_i}$ is said to intersect transversely at a point $x \in M$ if 
\begin{equation*}
\text{codim}(\bigcap_i T_xZ_i) = \sum_i \text{codim}(T_xZ_i).
\end{equation*}
} intersecting embedded hypersurfaces. Vector fields tangent to $Z$ form a Lie algebroid, called the \textbf{self-crossing $b$-tangent bundle}.
\end{definition}

\begin{proposition}\label{prop:logfib}
Let $Z := \cup_iZ_i$ the union of transversly intersecting embedded hypersurfaces, and let $\A_{Z_i}$ denote the b-tangent bundle associated to $Z_i$. Then:
\begin{equation*}
\A_Z = \A_{Z_1}~ {}_{\rho_1}\times_{\rho_2} \A_{Z_2}~ {}_{\rho_2}\times_{\rho_3} \cdots {}_{\rho_{k-1}}\times_{\rho_k} \A_{Z_k}.
\end{equation*}
\end{proposition}
The manifold $M$ can then be stratified into submanifolds $Z[a]$, meaning the intersection locus of $a$ distinct hypersurfaces $Z_{i_1},\cdots,Z_{i_a}$. We use $Z[0]$ to denote $M \setminus Z$.

\begin{definition}\label{def:ellx}
Let $(W_i,s_i)$ be elliptic divisors, with transversely intersecting critical sets $D_i$ and let $I_s$ be the product of the ideals $s_i(\Gamma(W_i^*))$. Vector fields preserving $I_s$ form a Lie algebroid, called the \textbf{self-crossing elliptic tangent bundle} denote $\A_{\abs{D}}$.
\end{definition}
\begin{proposition}\label{prop:ellfib}
Let $(W_i,s_i)$ be elliptic divisors be as above, and let $\A_{\abs{D_i}}$ denote their elliptic tangent bundles. Then
\begin{equation*}
\A_{\abs{D}} = \A_{\abs{D_1}}~ {}_{\rho_1}\times_{\rho_2} \A_{\abs{D_2}}~ {}_{\rho_2}\times_{\rho_3} \cdots {}_{\rho_{k-1}}\times_{\rho_k} \A_{\abs{D_k}}.
\end{equation*}
\end{proposition}

Geometric structures on the self-crossing $b$-tangent bundle are studied in \cite{GLPR17,MS18}, and geometric structures on the self-crossing elliptic tangent bundle in \cite{CKW20}.
\begin{remark}
The above notions are not the most general. One can ask that the structure only locally looks as a product of log/elliptic divisors, e.g. the figure eight in $\rr^2$. However, we will only consider examples which are globally products.
\end{remark}

%------------------------------------------------------------------
\section{Regularisation} \label{sec:regularisation}

We now present the main idea of this article. Suppose that $\A \to M$ is one of the Lie algebroids introduced in Section \ref{sec:examples}. We can then produce a Lie algebroid $\Fcal \to W$ of foliation type that submerses onto $\A$. We think of this as a desingularisation process. This will allow us (Subsection \ref{ssec:regularisationDistributions}) to lift geometric structures on $\A$ to geometric structures on $\Fcal$; the latter are easier to work with. Concrete applications to singular contact structures are given in Section \ref{sec:Weinstein}.

\subsection{The trivial regularisation in $b^k$-Geometry}

We begin with the regularisation of $b^k$-tangent bundles. We are interested in the following concept:
\begin{definition} \label{def:regularisation}
Let $M$ be a manifold, $Z \subset M$ a hypersurface, and $j \in J_Z^{k-1}$ a jet, and let $\A^k_Z$ denote the associated $b^k$-tangent bundle. Let $I$ be either $\RR$ or $\NS^1$. We consider the manifold $M \times I$ endowed with the $I$-action given by the translations on its second component.

An $I$-invariant codimension-$1$ foliation $\Fcal$ in $M \times I$ is said to be a \textbf{regularisation} if the following properties are satisfied:
\begin{enumerate}
\item There is a Lie algebroid submersion
\begin{center}
\begin{tikzcd}
\Fcal \ar[r,"\Psi"] \ar[d] & \A_Z \ar[d,"\rho"] \\
T(M \times I) \ar[r,"d\pi"] & TM.
\end{tikzcd}
\end{center}
\item The connected components of $Z \times I$ are leaves.
\item All other leaves are transverse to the $I$-direction and diffeomorphic to a connected component of $M \setminus Z$.
\end{enumerate}
The components of $Z\times I$ are called the \textbf{central leaves} of $\Fcal$.
\end{definition}
Do note that $\Psi$ is necessarily a fibrewise isomorphism between $T\Fcal$ and $\A^k_Z$, for dimension reasons. It follows that a regularisation is an Ehresmann connection that additionally satisfies certain triviality properties in terms of its holonomy.

We distinguish between the cases $I = \RR$ and $I = \NS^1$ by referring to the former as the \textbf{trivial regularisation} and to the latter as the \textbf{compact regularisation}. The following is immediate:
\begin{lemma}
$\A^k_Z \rightarrow M$ admits a trivial regularisation if and only if it admits a compact regularisation.
\end{lemma}

\subsubsection{Regularisations from defining functions}

Our goal now is to prove that regularisations exist. The key feature to be exploited is the local triviality provided by the jet $j \in J_Z^{k-1}$. Do note that, even if $k=1$, we work (for now) under the assumption that $Z$ is coorientable. In the next lemma we also work under assumptions of global triviality:
\begin{lemma} \label{lem:trivialRegularisation}
Let $M$ be a manifold. Let $Z \subset M$ be a hypersurface. Let $f: M \to \RR$ be a global defining function for $Z$. Let $\A^k_Z$ be the $b^k$-tangent bundle given by $j = j^{k-1}_Zf$. We write $s$ for the $\RR$-coordinate in $M \times \RR$.

The foliation $\Fcal_f := \ker(\theta_f)$ given by $\theta_f := df + f^k ds$ is a trivial regularisation.
\end{lemma}
\begin{proof}
The $\RR$-invariance of $\Fcal_f$ is immediate from the definition of $\theta_f$. Involutivity follows from the computation $\theta_f \wedge d\theta_f = 0$. For the first property, consider a point $p \in Z \times \RR$. Then we have that $(\Fcal_f)_p = T_p(Z\times \rr)$, showing that the connected components of $Z \times \rr$ are leaves.

For the second property we observe that, in the complement of $Z \times \RR$, the foliation is defined by the one-form $d((k-1)^{-1}f^{-k+1})-ds$ and hence it leaves are graphical.

We now address the last property. Away from $Z$, $\Psi$ is necessarily given by $d\pi$. By density it follows that there is at most one extension of $\Psi$ to $Z \times \rr$. We prove that such an extension exists by working locally. On a tubular neighbourhood $U$ of a point $p \in Z$ we can find local coordinates $(z,x_2,\ldots,x_n)$ in which $f=z$. It follows that $\A_Z^k$ is spanned locally by 
\[ \langle z^k\partial_z, \partial_{x_2},\ldots,\partial_{x_n} \rangle. \]
Similarly, we have that
\begin{equation*}
\Fcal_f = \inp{z^k\partial_z - \partial_s, \partial_{x_2},\ldots,\partial_{x_n}}
\end{equation*}
in $U \times \rr$. We then set $\Psi$ to be given by $\Psi(z^k\partial_z - \partial_s) := z^k\partial_z$ and by $d\pi$ when applied to the rest of the frame.
\end{proof}

\subsubsection{Now without a global function}

More generally, we consider the setting in which $Z$ is coorientable but not necessarily given by a global function:
\begin{lemma} \label{lem:trivialRegularisation2}
Let $M$ be a manifold. Let $Z$ be a coorientable hypersurface. Fix a jet $j \in J_Z^{k-1}$ with associated $b^k$-tangent bundle $\A^k_Z$. Then, $\A^k_Z$ admits a trivial regularisation.
\end{lemma}
\begin{proof}
On a tubular neighbourhood $U \supset Z$, choose a semilocal defining function $z$ for $Z$; we require $j^{k-1}_Zu = j$ (note that, for $k=1$, this still allows us to choose $z$ compatible with any coorientation of $Z$). Let $\chi : \Ucal \rightarrow \rr$ be a function which is 1 near $Z$ and 0 near the boundary of $\Ucal$ and only depends on $z$. Let $s$ denote the $\rr$-coordinate on $M\times \rr$. Then we set
\begin{align*}
\Fcal := 
\begin{cases}
\ker(\chi(z)dz + z^kds) \quad \text{on }\Ucal \\
\ker(ds) \quad \text{ elsewhere}.
\end{cases}
\end{align*}
Involutivity and $\RR$-invariance are clear by construction. For the rest of the properties we can argue as in the proof of Lemma \ref{lem:trivialRegularisation}.
\end{proof}

\begin{figure}
\includegraphics[scale=0.7]{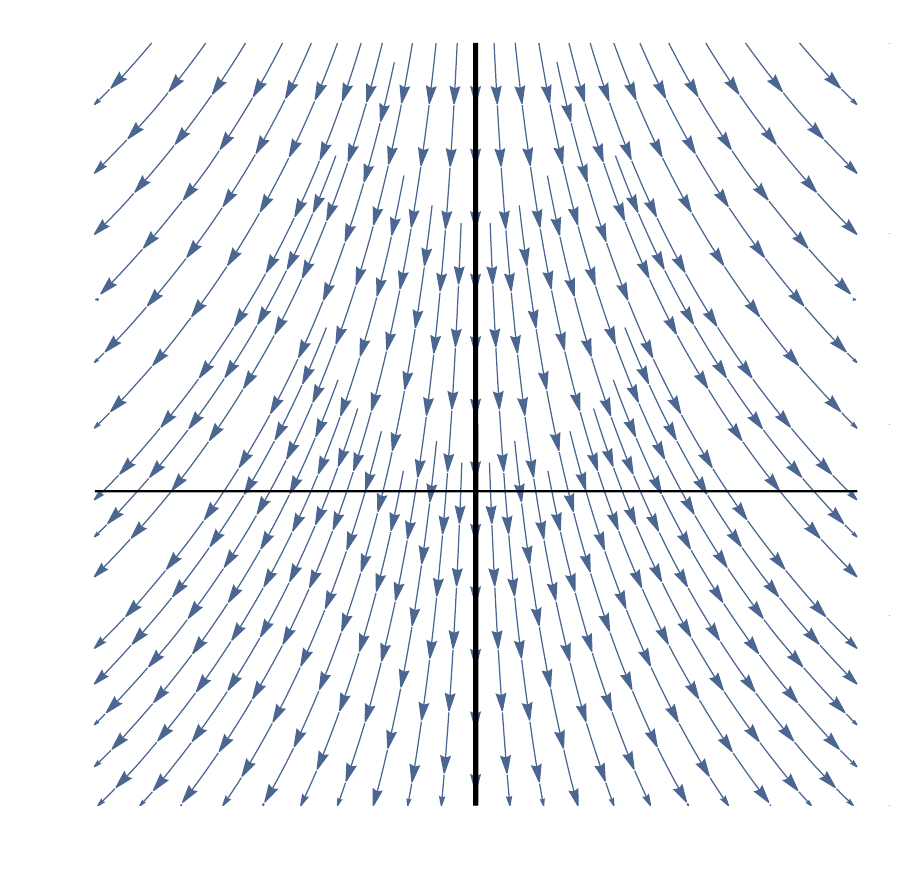}
\caption{When $M = \RR$, $Z = \set{0}$, $k=1$, and the defining function is the identity, the foliation on $\RR^2$ produced by Lemma \ref{lem:trivialRegularisation}, is the foliation defined by the vector field $z\partial_z - \partial_s$.}
\end{figure}

\subsubsection{Coorientations}

We remark:
\begin{lemma} \label{lem:coorientationsTrivialRegularisation}
Let $\A^k_Z \rightarrow M$ be the $b^k$-tangent bundle given by the hypersurface $Z \subset M$ and the jet $j \in J_Z^{k-1}$.
\begin{itemize}
\item If $k$ is even, a regularisation $\Fcal$ defines a coorientation $\tau_\Fcal$ of $Z$.
\item If $k$ is odd, a regularisation $\Fcal$ provides an orientation of the vertical bundle of $M \times \RR$.
\end{itemize}
\end{lemma}
\begin{proof}
We work on each component $Z_0 \subset Z$ separately. Consider a point $p \in Z_0$ and a neighbourhood $U \ni p$. On $U$, we can consider a defining function $z$ for $Z_0$, compatible with $j$. Due to coorientability, $Z_0$ is two-sided and $z$ marks which of the two sides is the positive one. This yields a coorientation $\tau_j$; this is not yet the coorientation claimed in the statement.

Now, $\A_Z^k$ is locally spanned by 
\[ \langle z^k\partial_z, \partial_{x_2},\ldots,\partial_{x_n} \rangle. \]
Given the Lie algebroid submersion $\Psi: \Fcal \rightarrow \A_Z^k$, we can consider the vector field $X = \Psi^{-1}(z^k\partial_z)$. Since it is non-vanishing, it must have non-zero vertical component.

If $k$ is even, and the vertical component of $X$ is negative, we set $\tau_\Fcal$ around $Z_0$ to be given by $\tau_j$. Otherwise we set $\tau_\Fcal = -\tau_j$. Equivalently, $Z_0$ is two-sided, and we set the positive side to be the one along which the leaves of $\Fcal$ spiral upwards. We claim that this definition does not depend on choices. The first choice is the local defining function $z$. However, the space of (local) defining functions for $Z_0$ and compatible with $j$ is convex, so all of them yield the same coorientation. The other choice is $j$ itself. Recall that the only jets defining the same Lie algebroid as $j$ are those of the form $\lambda j$, with $\lambda \in \RR^*$. Observe that the function $\lambda z$ is compatible with $\lambda j$. If we carry out the above construction for $\lambda j$, the resulting vector field is $\lambda^{k-1}X$. If $\lambda$ is positive, the vertical components of $X$ and $\lambda^{k-1}X$ have the same sign and $\tau_{\lambda j} = \tau_j$ holds. If $\lambda$ is negative, the vertical components of $X$ and $\lambda^{k-1}X$ have opposite signs and we also have $\tau_{\lambda j} = -\tau_j$. In either case we obtain $\tau_\Fcal$.

If $k$ is odd, we orient the vertical using $-X$. Geometrically, this means that the holonomy of $\Fcal$ along $Z_0$ is attracting as we move positively along the vertical. To prove that this is well-defined we reason as above; when proving independence of $j$, we see that the vertical component of $\lambda^{k-1}X$ has the same sign as $X$, proving the claim.
\end{proof}

We can combine Lemmas \ref{lem:trivialRegularisation2} and \ref{lem:coorientationsTrivialRegularisation} to show:
\begin{corollary} \label{cor:trivialRegularisation}
Let $M$ be a manifold. Let $Z$ be a coorientable hypersurface. Fix a jet $j \in J_Z^{k-1}$ with associated $b^k$-tangent bundle $\A^k_Z$. Then, $\A^k_Z$ admits a trivial regularisation such that:
\begin{itemize}
\item The coorientation it induces on $Z$ can be prescribed, if $k$ is even.
\item The orientation it induces on the vertical, over $Z$, can be prescribed, if $k$ is odd.
\end{itemize}
\end{corollary}
\begin{proof}
If $k$ is even, the regularisation $\Fcal$ given in Lemma \ref{lem:trivialRegularisation2} satisfies $\tau_\Fcal = \tau_j$. If $k$ is odd, the vertical orientation induced by $\Fcal$ agrees with the standard orientation of $\RR$. As a side remark, do observe that the same is true for the regularisation produced by Lemma \ref{lem:trivialRegularisation}.

In order to achieve other (co)orientations we can simply change the plus in the formula for the local $1$-form $\chi(z)dz + z^kds$ defining $\Fcal$. This is done for each component of $Z$ suitably.
\end{proof}

\subsubsection{Lack of uniqueness}

We will say that two trivial regularisations are \textbf{equivalent} if there is an $\RR$-equivariant diffeomorphism, fibered over $M$, taking one to the other. They are \textbf{isotopic} if the diffeomorphism may be assumed to be isotopic to the identity (through $\RR$-equivariant diffeomorphisms fibered over $M$). It follows from Corollary \ref{cor:trivialRegularisation} that:
\begin{corollary} \label{cor:nonUniqueness}
Let $M$ be a manifold. Let $Z$ be a coorientable hypersurface. Fix a jet $j \in J_Z^{k-1}$ and denote the corresponding $b^k$-tangent bundle by $\A$. Then, $\A$ admits two non-isotopic regularisations.
\end{corollary}
\begin{proof}
Use Corollary \ref{cor:trivialRegularisation} to construct a regularisation whose induced (co)orientations are opposite from the standard ones. Then the claim follows due to the invariance of (co)orientations under isotopy.
\end{proof}

Similarly:
\begin{corollary} \label{cor:nonUniquenessOdd}
Let $M$ be a manifold. Let $Z$ be a coorientable hypersurface. Fix a jet $j \in J_Z^{k-1}$ and denote the corresponding $b^k$-tangent bundle by $\A$. Suppose $k$ is odd and $Z$ has more than one connected component. Then, $\A$ admits two non-equivalent regularisations.
\end{corollary}
\begin{proof}
Using Corollary \ref{cor:trivialRegularisation} we construct a regularisation whose induced vertical orientation in each component of $Z$ is positive. We can also construct a regularisation whose induced vertical orientation is positive in all components but one. These two regularisations cannot be equivalent.
\end{proof}

For $k$ even, one can reason similarly and prove:
\begin{corollary} \label{cor:nonUniquenessEven}
Let $M$ be a manifold. Let $Z$ be a coorientable hypersurface. Fix a jet $j \in J_Z^{k-1}$ and denote the corresponding $b^k$-tangent bundle by $\A$. Suppose $k$ is even and consider the action of $\Diff(M,Z) \times \ZZ/2\ZZ$ on the set of coorientations of $Z$; here $\ZZ/2\ZZ$ acts by reversing all coorientations.

If this action is not transitive, $\A$ admits two non-equivalent regularisations.
\end{corollary}
(Co)orientations are the only invariant that we have used to distinguish regularisations. Our expectation is that there should be a non-trivial moduli, as suggested by the freedom we have in choosing the cut-off function in the proof of Lemma \ref{lem:trivialRegularisation2}.

We revisit the question of uniqueness in Subsection \ref{sssec:uniquenessIntrinsic} below.

\subsection{The intrinsic regularisation in $b$-Geometry}

In the particular case of $b$-geometry, we can present a regularisation recipe that is the intrinsic analogue of Lemma \ref{lem:trivialRegularisation}. It has the additional advantage of generalising to the situation in which the hypersurface $Z$ is not coorientable. This requires us to replace $M \times \RR$ by a non-trivial line bundle:
\begin{proposition} \label{prop:intrinsicRegularisation}
Let $(M,Z)$ be a b-pair, and let $(W,s)$ be the associated b-divisor. Then, there exists a coorientable foliation $\Fcal$ on $\pi: W^* \setminus M \rightarrow M$ such that:
\begin{itemize}
\item $\Fcal$ is $\RR^*$-invariant.
\item The connected components of $\pi^{-1}(Z)$ are leaves.
\item All other leaves are diffeomorphic to a connected component of $M \setminus Z$.
\item There is a Lie algebroid submersion
\begin{center}
\begin{tikzcd}
\Fcal \ar[r,"\Psi"] \ar[d] & \A_Z \ar[d,"\rho"] \\
T(W^* \setminus M) \ar[r,"d\pi"] & TM.
\end{tikzcd}
\end{center}
\end{itemize}
\end{proposition}
\begin{proof}
The function
\begin{align*}
\varphi : W^* \rightarrow \RR, \quad e_x \mapsto \inp{s(x),e_x},
\end{align*}
is regular outside the zero-section and therefore defines a foliation $\Fcal$ on $W^* \setminus M$. Using the local coordinate description of $\A_Z$ it becomes apparent that $\pi : W^* \setminus M \rightarrow M$ defines a Lie algebroid submersion from $\Fcal$ to $\A_Z$. Checking the other properties of $\Fcal$ is straightforward.
\end{proof}
The foliated manifold produced by Proposition \ref{prop:intrinsicRegularisation} is called the \textbf{intrinsic regularisation} of $\A_Z$. It is unclear to the authors whether there is a similar description in the $b^k$-setting.

\begin{figure}
\includegraphics[scale=0.7]{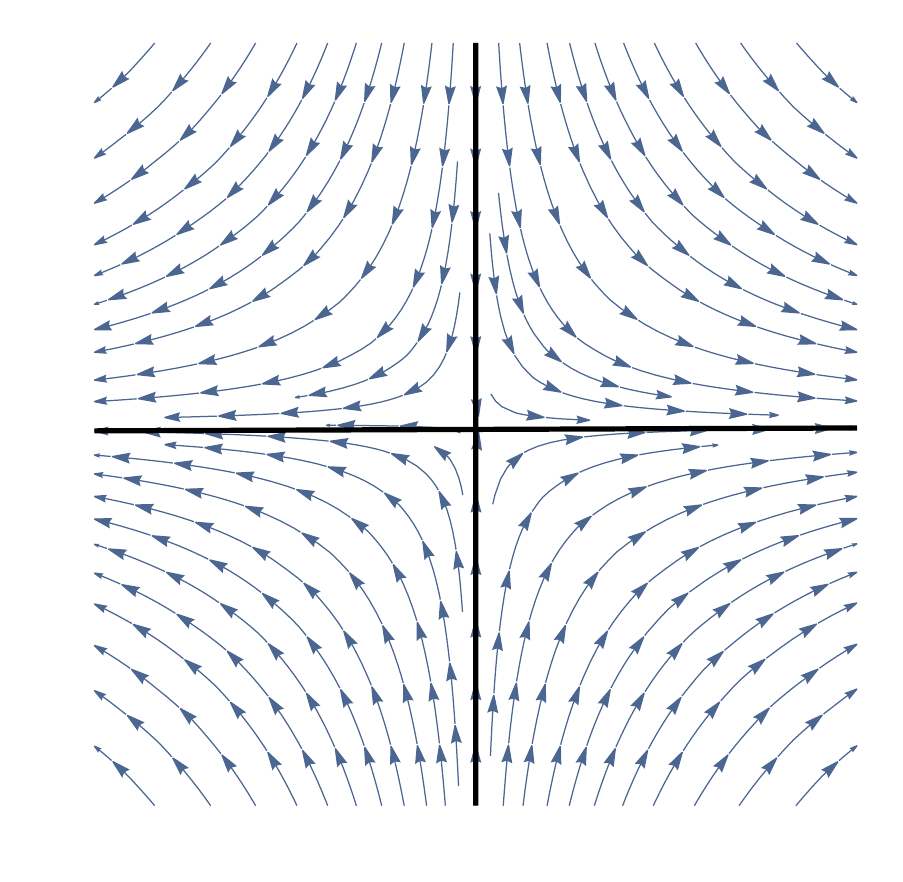}
\caption{When $M = \RR$ and $Z = \{0\}$, the foliation on $\RR^2\setminus\RR$ obtained using Proposition \ref{prop:intrinsicRegularisation} is given by the vector field $z\partial_z - s\partial_s$. Note that the foliation on the entirety of $\rr^2$ is singular, but when restricted to the complement of the $x$-axis it becomes regular.}
\end{figure}

\subsubsection{Uniqueness} \label{sssec:uniquenessIntrinsic}

We now explain how the intrinsic regularisation is actually intrinsic. In the statement of Proposition \ref{prop:intrinsicRegularisation} we are fixing some auxiliary data; namely, the section $s: M \rightarrow W$. Given some other section $s': M \rightarrow W$, also with transverse vanishing locus $Z$, we deduce that $h := s/s': M \rightarrow \RR$ is a non-vanishing function. Multiplication by $h$ is then an $\RR^*$-equivariant isomorphism of $W$ that takes $s'$ to $s$.

We then say that two foliations in $W^* \setminus \RR$ are \textbf{equivalent} if there is an $\RR^*$-equivariant diffeomorphism taking one to the other. The previous discussion implies that:
\begin{corollary}
Let $s,s': M \rightarrow W$ be sections whose transverse vanishing locus is $Z$. Then, the associated regularisations are equivalent.
\end{corollary}

Let us particularise to the case in which $W$ is the trivial bundle $M \times \RR$. We see $W$ as its own dual using the standard scalar product in $\RR$. We can then consider the positive cone $M \times \RR^{>0}$ and take exponential coordinates
\[ \Psi: M \times \RR \rightarrow M \times \RR^{>0} \]
fibrewise. The $\RR$-action in the source by translations corresponds then to the $\RR^{>0}$-action in the target by multiplication. Then:
\begin{proposition}
Let $(M,Z)$ be a b-pair with $Z$ given as the zero set of a global defining function $f$. Then, the exponential map $\Psi$ maps the trivial regularisation to the intrinsic regularisation (both associated to $f$).
\end{proposition}
\begin{proof}
For notational clarity, let us write $W^* = M \times \RR$. We see $f$ as a section of $W$, which we have identified with $W^*$. We use coordinates $(x,t)$ in $W^*$. Then, the intrinsic regularisation in $W^* \setminus M$, is given by the level sets of $tf(x)$.

We then consider $\Psi: M \times \RR \rightarrow W^*$ given by $(x,s) \mapsto (x,e^s)$. It follows that the pullback of the intrinsic regularisation is given by the level sets of $e^sf(x)$. Equivalently, it is the foliation given as the kernel of the $1$-form $e^s(fds+df)$, which is indeed the trivial regularisation.
\end{proof}
In particular, the regularisations produced by Lemma \ref{lem:trivialRegularisation} are unique up to $\RR$-equivariant equivalence.

%\begin{remark}
%\alert{integrate with connections!}
%There are more Lie algebroids present on $L^*$, besides $T\ff$. First the regular foliation on $L^*\backslash M$ extends to a singular foliation on $L^*$. The algebroid associated to this singular foliation, consists of all vector fields which annihilate the function $\varphi$ from the proof (\alert{not completely sure about this one}). 
%
%Moreover, another Lie algebroid can be defined as the vector fields which preserved the ideal $\pi^*I_s$. The ideal $\pi^*I_s$ is precisely the vanishing ideal of the union of the zero-section and $L|_Z$. This is an immersed hypersurface, but as the intersections are transverse one can still associate a Lie algebroid to it. 
%\end{remark}

\subsubsection{The additive and compact regularisations}

Consider a b-divisor $(W,s)$ and the corresponding intrinsic regularisation $(W^* \setminus M, \Fcal)$ produced by Proposition \ref{prop:intrinsicRegularisation}. Consider the $\ZZ_2$-action given by the $\RR^*$-action on $W$ and write $V \rightarrow M$ for the trivial $\RR^+$-bundle obtained upon quotienting. The $\RR^*$-invariance of $\Fcal$ implies that $V$ is endowed with an $\RR^+$-invariant foliation. This foliation can in turn be pulled back using the fibrewise exponential map $M \times \RR \rightarrow V$. We deduce:
\begin{corollary}
Any $b$-tangent bundle admits a trivial regularisation. 
\end{corollary}

We can take this a step further and consider the $(\ZZ_2 \times \ZZ)$-action on $W$ given by
\[ (\epsilon,n) \cdot v = \epsilon 2^n v. \]
The quotient of $W^* \setminus M$ by this action is then a trivial $\NS^1$-bundle. We deduce:
\begin{corollary}\label{cor:compactreg}
Any $b$-tangent bundle admits a compact regularisation.
\end{corollary}

%\begin{remark}[Linearisability]
%The foliation $\ff_{\reg,t}$ on $M\times \rr$ is linearisable around the central leaf $Z\times \rr$. Indeed, on a tubular neighbourhood of $Z \times \rr$ the map
%\begin{equation*}
%\varphi : Z \times \rr \times \rr \rightarrow Z \times \rr \times \rr, (x,z,s) \mapsto (x,e^sz,s),
%\end{equation*}
%provides a foliated diffeomoprhism between $\ff_{\reg,t}$ and the product foliation on $M \times \rr$. 
%
%Although the compact regularisation is often more useful in practice, it also has a downside. Note that the foliation $\ff_{\reg,c}$ is no longer linearisable around the central leaf. In fact, every leaf different from the central one ``turbulizes'' to the central leaf. This is depicted in Figure \ref{fig:compactreg}.
%\end{remark}
%\begin{figure}\label{fig:compactreg}
%\includegraphics[scale=0.7]{Pictures/compactreg.pdf}
%\caption{The (\alert{horrible}) picture of the compact regularisation of $M = \RR$ with $Z = \{0\}$. The foliation on $\rr \times S^1$ is given by the vector field $z\partial_z - \partial_{\theta}$. The central leaf, $\set{z=0}$, is an $S^1$ and all other leaves are open and spiral towards it.}
%\end{figure}

\subsubsection{Regularisation as a connection}

Given a b-divisor $(W,s)$ associated to the b-pair $(M,Z)$, recall from Remark \ref{rem:logconn} that $W$ is isomorphic to $Q_{\A_Z}= \det(\A^*)\otimes \det(TM)$. We now prove that the intrinsic regularisation can be interpreted as the Ehresmann connection associated to the canonical Lie algebroid representation.

\begin{proposition}\label{prop:conrelation}
Let $(M,Z)$ be a b-pair, and consider the canonical $\A_Z$-representation $(Q_{\A_Z},\nabla)$. Then, the associated $\A_Z$-Ehresmann connection is precisely the intrinsic regularisation $\ff$ associated to the b-divisor $(Q_{\A_Z}^*,\det(\rho_{\A_Z})^*)$.
\end{proposition}
\begin{proof}
We argue by computing the local one-forms of the canonical representation. Using the normal form of the b-tangent bundle, one obtains that a local frame for $Q_{\A_Z}$ is given by:
\begin{equation*}
q = d\log z_{\alpha} \wedge dx_2 \wedge \cdots \wedge dx_n \otimes \partial_{z_{\alpha}} \wedge \partial_{x_2} \wedge \cdots \wedge \partial_{x_n} 
\end{equation*}
Using the expression of the canonical connection in Definition \ref{def:canonicalRep}, one-obtains that the only b-vector field acting non-trivially on $q$ is $z_{\alpha}\partial_{z_{\alpha}}$. Because $\nabla_{z_{\alpha}}\partial_{z_{\alpha}}(q) = -1$, we conclude that the canonical $A_Z$-representation has connection matrices $-d\log z_{\alpha}$. 

Let $\hor : \A_Z \rightarrow T(Q_{\A_Z}\backslash M)$, denote the horizontal lift with respect to this $\A_Z$ connection. In terms of the connection matrices the horizontal lift is given by $\hor(X) = X - d\log z_{\alpha} (X)s\partial_s$.

In the same local coordinates, the function $\varphi$ used to define the intrinsic regularisation is given by $z_{\alpha}s$. From this description it follows that $\hor(\A_Z) = \ff$.
\end{proof}
Figure 2 neatly describes this point of view. The vector field $z\partial_z$ is lifted to the vector field $z\partial_z -s\partial_s$. Notice that although this vector field becomes vertical at the central fibre, the horizontal part, $s\partial_s$, remains non-zero when views as an element of $p^!\A_Z$ and thus corresponds to an honest $\A_Z$-connection.

\begin{remark}
One can also consider the canonical representation of the $b^k$-tangent bundle and hope that the associated $\mathcal{A}_Z^k$-Ehersmann connection defines a foliation away from the zero-section. However, this is unfortunately not the case. If $\Gamma(\mathcal{A}_Z^k) = \inp{z^k\partial_{z}}$ on $\rr$, then the lift with the associated canonical representation is $z^k\partial_z - z^{k-1}ks\partial_s$, which vanishes at $\set{z=0}$. 

The existence of the regularisation of the $b$-tangent bundle arises from the fact that the horizontal distribution associated to the canonical $\A_Z$-representation is regular away from the zero-section. With the zero-section included, the horizontal distribution defines a singular foliation on $W^*$.
\end{remark}

\subsection{Regularisation for elliptic}

We can easily extend the regularisation to some other singularities. The regularisation for elliptic divisors is the complex analogue:
\begin{proposition}\label{prop:ellireg}
Let $(L,\sigma)$ be a complex log divisor with zero locus $D$, and let $(W,s)$ denote the corresponding elliptic divisor. Then, there exists a (real) codimension-two foliation $\Fcal$ on $\pi: L^* \backslash M \rightarrow M$ such that
\begin{itemize}
\item $\Fcal$ is $\CC^*$-invariant.
\item Connected components of $\pi^{-1}(D)$ are leaves.
\item All other leaves are graphical over and diffeomorphic to a component of $M \setminus D$.
\item There is a Lie algebroid submersion $\Fcal \rightarrow \A_{\abs{D}}$ which is a fibrewise isomorphism.
\end{itemize}
%Consequently if $M$ is elliptic-symplectic, then $L^*\setminus M$ admits a codimension-two symplectic foliation.
\end{proposition}
\begin{proof}
Consider the function
\begin{equation*}
\varphi : L^* \rightarrow \cc, \quad e_x \mapsto \inp{\sigma(x),e_x}.
\end{equation*}
As in the real case, one can show that the foliation on $L^*\backslash M$ by the level sets of $\varphi$ provides the desired foliation.
\end{proof}

\begin{remark}\label{rem:elliconrelation}
Suppose we are in the setting of Proposition \ref{prop:ellireg}. As explained in Remark \ref{rem:complexlogconn}, we may choose $(L,\sigma) = (Q_{\A_D}^*,\det(\rho_{\A_D})^*)$. One may argue exactly as in Proposition \ref{prop:conrelation} to show that the associated $\A_D$-Ehresmann connection coincides with $\Fcal$ from Proposition \ref{prop:ellireg}. 
\end{remark}

Taking a $\ZZ^2$-quotient of $L^*$, one obtains a codimension two foliation on a $T^2$-bundle over $M$.
\begin{proposition}
Let $(L,\sigma)$ be a complex log divisor. Then there exists a codimension-two foliation $\Fcal$ on $\pi: \SS^1L^* \times \SS^1 \rightarrow M$ such that
\begin{itemize}
\item $\Fcal$ is $T^2$-invariant.
\item $\pi^{-1}(D)$ is a leaf.
\item All other leaves are graphical over and diffeomorphic to a component of $M \setminus D$.
\item There is a Lie algebroid submersion $\Fcal \rightarrow \A_{\abs{D}}$ which is a fibrewise isomorphism.
\end{itemize}
\end{proposition}
This statement appears in the context of stable generalized complex structures as part of Proposition 2.17 in \cite{CG17}.

%In general, the central leaf can be any $T^2$-bundle over $D$, however in three dimensions we have:
%\begin{lemma}\label{lem:ellcentleaf}
%Let $(L,\sigma)$ be a complex log divisor an a compact three-dimensional manifold. Then, the central leaf of the regularisation is a disjoint union of tori.
%\end{lemma}
%\begin{proof}
%Because the normal bundle of $D$ in $M$ is given by the restriction of $L$, which is a complex line bundle, it must necessarily be trivial. The central leaf, $\pi^{-1}(D)$ is given by $\rest{L}{D}/\zz^2$, and is thus a copy of $D \times T^2$. Because $D$ is compact and one-dimensional it must be a disjoint union of circles, which finishes the proof.
%\end{proof}

\subsection{Regularisation for self-crossings}

The above regularisations can be adapted to the simple normal-crossing case. 
\begin{proposition}\label{prop:reglogx}
Let $Z_1,\ldots,Z_k$ be transversely intersecting embedded hypersurfaces on a manifold $M$. Let $Z$ denote their union, and let $(W_1,s_1),\ldots,(W_k,s_k)$ denote corresponding $b$-divisors. Denote by $\pi : W := W_1 \oplus \cdots \oplus W_k \rightarrow M$ the direct sum and consider the associated intrinsic regularisations $\ff_{j}$ on $W_j^*\backslash M$. Then,
\begin{equation*}
\ff := \ff_1~ {}_{\pi_1} \times_{\pi_2} \ff_2 ~{}_{\pi_2}\times_{\pi_3} \cdots {}_{\pi_{k-1}}\times_{\pi_k} \ff_k.
\end{equation*}
defines a codimension-$k$ foliation on $W^* \backslash M$ such that 
\begin{itemize}
\item $\Fcal$ is $(\RR^*)^k$-invariant.
%\item $\pi^{-1}(Z[i])$ is endowed with a codimension-$(k-i)$ foliation transverse to the vertical. When restricted to a leaf, $\pi$ is a fibre bundle, with $i$-dimensional fibres, over some component of $Z[i]$.
\item The map $d\pi: \Fcal \to \A_Z$ is a Lie algebroid submersion, and a fibrewise isomorphism.
\end{itemize}
\end{proposition}
\begin{proof}
As each $\ff_i$ is an $\rr^*$-invariant foliation on $W_i^*\backslash M$, the fibre product $\ff$ will be an $(\rr^*)^k$-invariant foliation.
By Proposition \ref{prop:logfib} Lie algebroid $\A_Z$, decomposes as a fibre-product:
\begin{equation*}
\A_Z = \A_{Z_1}~ {}_{\rho_1}\times_{\rho_2} \A_{Z_2}~ {}_{\rho_2}\times_{\rho_3} \cdots {}_{\rho_{k-1}}\times_{\rho_k} \A_{Z_k}
\end{equation*}
As each of the $\pi_i$ defines a Lie algebroid submersion $\ff_i \rightarrow \A_{Z_i}$, the map $\pi$ will define a Lie algebroid submersion onto $\A_Z$. By rank considerations, this is furthermore a fibre-wise isomorphism.
\end{proof}
%% TODO: prove that there is something like a b foliation with normal crossings and one can iterate the usual regularisation until one produces this.

And again, the complex analogue is given by:
\begin{proposition}\label{prop:regellx}
Let $I_{\abs{D_1}},\ldots,I_{\abs{D_k}}$ be co-orientable elliptic divisors on $M$ with transversely intersecting vanishing loci and let $I_{\abs{D}}$ be their product. Let $(L_i,\sigma_i)$ denote the associated complex log divisors and denote by $\pi : L:= L_1\oplus \cdots \oplus L_k \rightarrow M$, the direct sum and consider the associated intrinsic regularisations $\ff_j$ on $L_j^*\backslash M$.  Then 
\begin{equation*}
\ff := \ff_1~ {}_{\pi_1} \times_{\pi_2} \ff_2 ~{}_{\pi_2}\times_{\pi_3} \cdots {}_{\pi_{k-1}}\times_{\pi_k} \ff_k.
\end{equation*}
defines a codimension $2k$ foliation on $L\backslash M$ such that:
\begin{itemize}
\item $\Fcal$ is $(\CC^*)^k$-invariant.
%\item $\pi^{-1}(Z[i])$ is endowed with a codimension-$(2k-2i)$ foliation transverse to the vertical. When restricted to a leaf, $\pi$ is a fibre bundle, with $2i$-dimensional fibres, over some component of $Z[i]$.
\item The map $d\pi: \Fcal \to \elli$ is a Lie algebroid submersion, and a fibrewise isomorphism.
\end{itemize}
\end{proposition}
\begin{proof}
As each $\ff_i$ is an $\cc^*$-invariant foliation on $L_i^*\backslash M$, the fibre product $\ff$ will be an $(\cc^*)^k$-invariant foliation.

By Proposition \ref{prop:ellfib}, the Lie algebroid $\elli$ decomposes as a fibre-product:
\begin{equation*}
\A_{\abs{D}} = \A_{\abs{D_1}}~ {}_{\rho_1}\times_{\rho_2} \A_{\abs{D_2}}~ {}_{\rho_2}\times_{\rho_3} \cdots {}_{\rho_{k-1}}\times_{\rho_k} \A_{\abs{D_k}}.
\end{equation*}
As each of the $\pi_i$ defines a Lie algebroid submersion $\ff_i \rightarrow \A_{\abs{D_i}}$, the map $\pi$ will define a Lie algebroid submersion onto $\A_Z$. By rank considerations, this is furthermore a fibre-wise isomorphism.
\end{proof}
Because the tangent space to a fibre product, is the fibre product of the tangent spaces, we find that the leaves of the above foliation $\ff$ are of the form $L_1~ {}_{\pi_1} \times_{\pi_2} L_2 ~{}_{\pi_2}\times_{\pi_3} \cdots {}_{\pi_{k-1}}\times_{\pi_k} L_k$, with $L_i$ a leaf of $\ff_i$.

%\begin{proof}
%Consider the function
%\begin{equation*}
%\varphi : W \rightarrow \cc^k, \quad e_x \mapsto \inp{\sigma_1(x) \otimes \cdots \sigma_k(x),e_x}.
%\end{equation*}
%The desired foliation on $W \backslash M$ is again obtained as the level sets of this function.
%\end{proof}
In both cases one can  take quotients to obtain compact regularisations.

\begin{remark}
Suppose we are in the setting of Proposition \ref{prop:reglogx}, with $(W_i,s_i)= (Q_{\A_{Z_i}}^*,\det(\rho_{\A_{Z_i}})^*)$. We may take the direct sums of the canonical representations, to obtain an $\A_Z$-representation:
\begin{equation*}
(Q_{\A_{Z_1}}\oplus \cdots \oplus Q_{\A_{Z_k}},\nabla_1\oplus \cdots \oplus  \nabla_k)
\end{equation*} 
One may now apply Proposition \ref{prop:conrelation}, to show that the associated $\A_Z$-Ehresmann-connection is precisely the foliation $\ff$ from Proposition \ref{prop:reglogx}.

Similarly, suppose we are in the setting of Proposition \ref{prop:reglogx}, with $(L_i,\sigma_i)= (Q_{\A_{D_i}}^*,\det(\rho_{\A_{D_i}})^*)$. We may take the direct sums of the canonical representations, to obtain an $\A_D$-representation:
\begin{equation*}
(Q_{\A_{D_1}}\oplus \cdots \oplus Q_{\A_{D_k}},\nabla_1\oplus \cdots \oplus  \nabla_k)
\end{equation*} 
One may now apply Remark \ref{rem:elliconrelation}, to show that the associated $\A_D$-Ehresmann-connection is precisely the foliation $\ff$ from Proposition \ref{prop:regellx}.
\end{remark}

\subsection{Regularisation of Lie algebroid distributions} \label{ssec:regularisationDistributions}

The general philosophy of this paper is as follows:
\begin{center}
\fbox{\textbf{Geometric structures on $\A$ can be studied as foliated structures on its regularisation $\Fcal$.}}
\end{center}

A first result in this direction reads:
\begin{proposition} \label{prop:regularisationDistributions}
Let $\A \rightarrow M$ and $\Bcal \rightarrow N$ be Lie algebroids of the same rank, endowed with a Lie algebroid submersion
\begin{center}
\begin{tikzcd}
\Bcal \ar[r,"\Psi"] \ar[d] & \A \ar[d] \\
TN \ar[r,"d\pi"] & TM.
\end{tikzcd}
\end{center}
Given any regular distribution $\xi \subset \A$ we can consider its preimage $\Psi^*\xi \subset \Bcal$. Then:
\begin{itemize}
\item $(\Psi^*\xi)_i = \Psi^*(\xi_i)$, where $(\xi_i)_i$ is the Lie flag of $\xi$.
\item $\Psi^*\xi$ is bracket-generating if and only if $\xi$ is.
\item $\Psi^*\xi$ is involutive if and only if $\xi$ is.
\end{itemize}
\end{proposition}
The proof is immediate from the fact that $\Psi$ is a fibrewise isomorphism that preserves the Lie bracket.

\subsubsection{The $b$-case}

Proposition \ref{ssec:regularisationDistributions} deals with the lifting process to flat Ehresmann connections. The regularisation is certainly an example, but it has additional structure. We spell this out for regular $b$-distributions:
\begin{proposition}
Let $\A_Z \rightarrow M$ be a $b$-tangent bundle with canonical representation $\nabla$ and intrinsic regularisation $\mathcal{F}$. Let $\xi \subset \A_Z$ be a regular distribution. Then $\hor^{\nabla}(\xi) \subset \Fcal$ is a regular $\RR^*$-invariant distribution with $\hor^{\nabla}(\xi)_i = \hor^{\nabla}(\xi_i)$ and the following diagram commutes:
\begin{center}
\begin{tikzcd}
\xi_i \times \xi_j \ar[r,"\hor^{\nabla}"] \ar[d,"\Omega_{i,j}(\xi)"] & \hor^{\nabla}(\xi_i) \times \hor^{\nabla}(\xi_j) \ar[d,"\Omega_{i,j}(\hor^{\nabla}(\xi))"]\\
\xi_{i+j}/\xi_{i+j-1} \ar[r,"\hor^{\nabla}"] & \hor^{\nabla}(\xi_{i+j})/\hor^{\nabla}(\xi_{i+j-1})
\end{tikzcd}
\end{center}
In particular, $\xi$ is bracket generating if and only if its lift $\hor^{\nabla}(\xi)$ is bracket generating.
\end{proposition}
\begin{proof}
This follows readily from the fact that the canonical $\A_Z$-connection is flat, and the fact that $\Fcal$ is precisely its Ehresmann connection by Proposition \ref{prop:conrelation}.
\end{proof}

\subsubsection{The $b^k$-case}

Identically, we can make the following statements:
\begin{proposition}\label{prop:inducedcontactfol}
Let $(M,Z,\xi)$ be a regular $b^k$-distribution with $Z$ coorientable. Then:
\begin{itemize}
\item In the trivial regularisation, there is an $\rr$-invariant regular distribution lifting $\xi$.
\item In the compact regularisation, there is an $\SS^1$-invariant regular distribution lifting $\xi$.
\end{itemize}
Furthermore, taking the Lie flag and curvatures commute with passing to the regularisation.
\end{proposition}
Analogous statements hold in the complex-log, elliptic, and self-crossing settings.

%------------------------------------------------------------------------
\section{Lie algebroid contact structures} \label{sec:contact}

Before we get into general distributions on Lie algebroids, let us discuss the concrete case of contact structures. There is an already existing body of literature dealing with them in the algebroid setting, paticularly within the frameworks of Foliation Theory \cite{CPP14,dPP18} and $b^k$-Geometry \cite{MO18,MO21}.

Our goal in this section is to define them in general, introduce a number of standard constructions (Subsections \ref{ssec:symplectisation}, \ref{ssec:contactElements}, and \ref{ssec:contactHypersurface}), and then particularise to the settings of $b^k$-Geometry (Subsection \ref{ssec:bkContact}) and elliptic geometry (Subsection \ref{ssec:ellipticContact}).

\subsection{The definition}

\begin{definition}
Let $\A \to M$ be a Lie algebroid of rank $2r+1$. A corank-$1$ distribution $\xi \subset \A$ is said to be an \textbf{$\A$-contact structure} if its first curvature
\[ \Omega_{1,1}(\xi): \xi \times \xi \longrightarrow \A/\xi \]
is a non-degenerate $2$-form.
\end{definition}
Non-degeneracy implies, in particular, that $\xi$ is bracket-generating of step $2$.

As per usual, the contact condition can be rephrased in terms of forms:
\begin{definition}
Let $\A \to M$ be a Lie algebroid of rank $2r+1$. An \textbf{$\A$-contact form} $\alpha \in \Omega^1(\A)$ is a Lie algebroid 1-form such that
\[ \alpha \wedge (d\alpha)^r \in \Omega^{2r+1}(\A) \]
is a Lie algebroid volume form.
\end{definition}
Given an $\A$-contact form $\alpha$, we have that $\xi = \ker(\alpha) \subset \A$ defines an $\A$-contact structure. This follows from the formula 
\[ d\alpha = - \alpha \circ \Omega_{1,1}(\xi). \]
Moreover, any $\A$-contact structure is locally given as the kernel of an $\A$-contact form, but this need not be the case globally.

\begin{example}\label{ex:Heiscont}
Let $\mathcal{A} = \mathcal{H}_r$ be the Heisenberg Lie algebra, which has generators $p_1,\ldots,p_r,q_1,\ldots,q_r,z$ satisfying $[p_i,q_j] = \delta_{ij}z$ and all other brackets zero. Then $\xi = \text{span}(p_1,\ldots,p_r,q_1,\ldots,q_r)$ is a contact structure.
\end{example}

\subsection{Recap: Algebroid symplectic forms} \label{ssec:algebroidSymplectic}

Contact structures can be understood as homogeneous versions of symplectic structures, via the symplectisation functor (see Subsection \ref{ssec:symplectisation} below). See \cite{TYV} for a more general incarnation of this phenomenon. We now recall some of the basic theory of symplectic forms in Lie algebroids.
\begin{definition}
Let $\A \to M$ be a Lie algebroid of rank $2r$. An \textbf{$\A$-symplectic form} $\omega \in \Omega^2(\A)$ is a closed Lie algebroid 2-form for which $\omega^r \in \Omega^{2r}(\A)$ is a Lie algebroid volume form.
\end{definition}

\subsubsection{The Liouville form}

Cotangent bundles are the prototypical examples of symplectic manifolds. It turns out that, similarly, the dual of any Lie algebroid admits a tautological $1$-form whose differential is symplectic. Let us go through the construction; see also \cite{LMM04,Smi21}.

Let $\A \rightarrow M$ be a Lie algebroid and write $\pi: \A^* \to M$ for its dual. We can then consider the pullback algebroid $\pi^!\A \rightarrow \A^*$. Do note that, if $\A$ is one of the algebroids presented in Section \ref{sec:examples}, then $\pi^!\A$ is of the same type.
\begin{definition} \label{def:Liouville}
The \textbf{canonical/Liouville one-form} $\lambda_\can \in \Omega^1(\pi^!\A)$ is defined as
\begin{equation*}
\lambda_\can(\alpha)(v) = \alpha(d_{\A}\pi(v)),
\end{equation*}
for all $\alpha \in \A^*$ and $v \in (\pi^!\A)_{\alpha}$.
\end{definition}

Let $(v_1,\ldots,v_r)$ be a local frame of $\A$, let $(\alpha_1,\ldots,\alpha_r)$ be the dual coframe of $\A^*$, and denote by $(t_1,\ldots,t_r)$ the associated fibrewise coordinates on $\A^*$. Then, the canonical one-form reads:
\begin{equation*}
\lambda_{\can} = t_1\alpha_1 + \cdots + t_r \alpha_r.
\end{equation*}
Using this local expression is immediate to check that:
\begin{lemma}[\cite{LMM04}]
The two-form $\omega_{\can} = -d\lambda_{\can} \in \Omega^2(\pi^!\A)$ is symplectic.
\end{lemma}

\subsubsection{The $b$-setting}

We can now particularise to $b$-tangent bundles. Recall:
\begin{lemma}[\cite{GMP14}]\label{lem:logsympcosymp}
Let $(M,Z,\omega)$ be a b-symplectic manifold. Then, there is an induced cosymplectic structure $(\theta,\eta)$ on $Z$. Here $\theta = \Res(\alpha)$, and $\eta = \iota^*_Z(\omega - d\log \abs{z} \wedge p^*\theta)$, where $\abs{z}$ is the distance to $Z$ with respect to some metric.
\end{lemma}
Note that $\theta$ is well-defined, but $\eta$ will depend on the choice of metric.

The fact whether or not $M$ and $Z$ are orientable plays an important role in our study:
\begin{lemma}
Let $(M,Z,\omega)$ be a b-symplectic manifold. There exists a global defining function for $Z$ if and only if $M$ is orientable.
\end{lemma}
\begin{proof}
Let $(W,s)$ be the divisor associated to $Z$. According to Remark \ref{rem:logconn} $\det(M) \cong W$, since $\omega^n$ trivialises $\det(A_Z^*)$.
%% TODO: check this
\end{proof}

Consider the following concrete example: Let $(A,B)$ be a b-pair. Then, the Lie algebroid $\pi^!\A_B \rightarrow \A_B^*$ is the b-tangent bundle associated to the b-pair $(\A_B^*, \pi^{-1}(B))$. Furthermore, we can endow it with the canonical $2$-form $\omega_\can$. It follows that $\pi^{-1}(B)$ has a cosymplectic structure.

\subsection{The symplectisation} \label{ssec:symplectisation}
 
We now discuss the symplectisation functor. It has appeared already in the literature in the foliated \cite{dPP18} and $b$-settings \cite{MO18}. We will need the following notation: if $E_i \rightarrow M_i$ are vector bundles, and $\pi_i : M_1\times M_2 \rightarrow M_i$ denote the projections, then $E_1\boxtimes E_2$ is the vector bundle $\pi_1^*E_1 \oplus \pi_2^* E_2$.

We define:
\begin{definition}
Let $\alpha \in \Omega^1(\A)$ be an $\A$-contact form on $M$. Consider the Lie algebroid $\A \boxtimes T\RR \rightarrow M \times \RR$ and write $t$ for the coordinate in the $\RR$ factor.

The pair $(\A \boxtimes T\RR \rightarrow M \times \RR, d(e^t\alpha))$ is the \textbf{symplectisation} of $(\A \to M,\alpha)$.
\end{definition}
It can be checked that $d(e^t\alpha)$ is a $\A \boxtimes T\RR$-symplectic form.

We can also observe that the symplectisation is a functor. It takes Lie algebroid contact structures of a given subclass (e.g. $b^k$, complex-log, elliptic, self-crossing geometries) to symplectic structures of the same subclass. Indeed, any morphism $\psi: (\A \to M) \rightarrow (\Bcal \to N)$ (meaning a bracket-preserving bundle map between algebroids, commuting with the anchor, and lifting an equi-dimensional immersion) pulling back the contact structure $\eta = \ker(\beta) \subset \Bcal$ to $\xi = \ker(\alpha) \subset \A$ has a conformal factor $f: M \rightarrow \RR$ given by $\psi^*\beta = e^f\alpha$. The associated symplectomorphism between the symplectisations is then $(p,t) \mapsto (\psi(p),e^{f(p)}+t)$.

\subsection{The space of contact elements} \label{ssec:contactElements}

Given a Lie algebroid $\A \rightarrow M$, we now explain how to produce a Lie algebroid contact structure in $\SS\A^*$, the sphere bundle of its dual. This generalises the usual construction of the space of contact elements in the sphere cotangent bundle.

Let $\RR^*$ act on $\pi: \A^*\setminus M \rightarrow M$ by scaling and consider the quotient $\SS\A^* := (\A^*\setminus M)/\RR^*$. Because $\RR^*$ acts via $\pi^!\A$-algebroid morphisms, $\pi^!\A$ descends to an algebroid on $\SS\A^*$. Alternatively, the same Lie algebroid may be obtained as the pullback $\pi_\SS^!\A$, where $\pi_\SS : \SS\A^* \rightarrow M$ denotes the projection. Since $\lambda_{\can} \in \Omega^1(\pi^!\A)$ is $\RR^*$-homogeneous, its kernel is $\RR^*$-invariant and thus descends to a corank-$1$ distribution $\Dcal_\can \subset \pi_\SS^!\A$.
\begin{lemma} \label{lem:contactElements}
The canonical $\pi_\SS^!\A$-distribution $\Dcal_\can$ is contact. Its symplectisation is isomorphic to $(\pi^!\A|_{\A^*\setminus M},\omega_\can)$.
\end{lemma}
The pair $(\pi_\SS^!\A, \Dcal_\can)$ is called the \textbf{space of contact elements} of $\A \to M$.

\subsection{Contact hypersurfaces} \label{ssec:contactHypersurface}

Another source of $\A$-contact structures are hypersurfaces in $\A$-symplectic manifolds.
\begin{definition}
Let $\rho: \A \rightarrow M$ be a Lie algebroid with symplectic form $\omega \in \Omega^2(\A)$. A hypersurface $N \subset M$ is said to be \textbf{contact} if there are a tubular neighbourhood $U \supset N$ and a section $X \in \Gamma(\A|_U)$ such that $\rho(X)$ is transverse to $N$ and $\mathcal{L}_X\omega = \omega$ on $\U$.
\end{definition}

Then, as in the classic case we can prove:
\begin{lemma}\label{lem:contacthypersurface}
Let $\rho: \A \rightarrow M$ be a Lie algebroid with symplectic form $\omega \in \Omega^2(\A)$. Let $j: N \to M$ be the inclusion of a contact hypersurface.

Then $j^*(\iota_X\omega)$ is a $j^*\A$-contact form.
\end{lemma}

\subsection{The $b$-setting} \label{ssec:bkContact}

For future reference we collect some results on contact forms on the b-tangent bundle. Most of these have appeared before, but we will expand on some of them.

\subsubsection{Geometric structure on the critical set}

The geometric structure induced on the singular set of a b-contact structure is described in \cite{MO18} by passing through the Jacobi structure associated with it. We will give a more detailed description of this structure, and relate it to the regularisation. 

\begin{theorem}[\cite{MO18}]\label{th:geomCE}
Let $(M^{2n+1},\alpha)$ be a $b^k$-contact manifold, and let $u := \Res(\alpha)$. Then:
\begin{itemize}
\item $u^{-1}(\set{0})$ inherits a contact structure.
\item $Z\backslash u^{-1}(\set{0})$ inherits a locally conformal symplectic structure.
\end{itemize}
\end{theorem}
%The condition $R_Z \in \im \Lambda_Z^{\sharp}$ can also be phrased in terms of the one-form:
%\begin{lemma}
%Let $(M^{2n+1},\alpha)$ be a $b^k$-contact manifold and let $(\Lambda,R)$ be the associated Jacobi structure. Then $(R_Z)_p \not \in \im \Lambda_p^{\sharp}$ if and only if $\Res(\alpha)(p) = 0$.
%\end{lemma}
%\begin{proof}
%From the definition it follows that $\Res(\alpha)(p) = 0$ if and only if $R_p$ has no $t\partial_t$ component. But as $d_p\alpha|_{\im \Lambda_p^{\sharp}}$ is an isomorphism, and $R_p = R_{Z,p}$ when $\Res(\alpha) = 0$ we must have $R_Z \not\in \im \Lambda_{Z,p}^{\sharp}$.
%\end{proof}

The geometric structure on the critical set may also be obtained in the following way. Write $\alpha = \tilde{u}dz/z^k + \tilde{\beta}$ with $\tilde{u} \in C^{\infty}(M)$ and $\beta \in \Gamma(p^*TZ)$, and let $(u,\beta) := (\Res(\alpha),\iota_Z^*\beta)$. The contact condition ensures that $u$ vanishes transversely, and thus $\Gamma := u^{-1}(\set{0})$ defines a hypersurface in $Z$. The contact condition furthermore implies that $\alpha_{\Gamma} :=\iota^*_\Gamma\beta$ defines a contact structure on $\Gamma$. On the complement of $\Gamma$, we consider the two-form $\tau := d\beta + u^{-1}\beta \wedge du $. This defines a locally conformally symplectic form. Moreover, we actually see that $\omega := u^{-1}\tau= d(u^{-1}\beta)$ is exact, and we can conclude on the following improvement of Theorem \ref{th:geomCE}, which will be important in the discussion on Weinistein conjectures in Section \ref{sec:Weinstein}:
\begin{proposition}\label{prop:bettergeom}
Let $(M^{2n+1},Z,\alpha)$ be a $b^k$-contact manifold. Then $Z$ inherits the following geometric structure:
\begin{itemize}
\item The complement $Z \backslash \Gamma$ inherits the structure of an ideal Liouville domain\footnote{As in Definition \ref{def:idealdomain}. Note that here we view $Z$ as the ideal boundary of $M\backslash Z$.} $\omega := d(u^{-1}\beta) \in \Omega^2(Z\backslash \Gamma)$.
\item The hypersurfaces $\Gamma \subset Z$ given by $u^{-1}(\set{0})$ inherit a contact form $\iota_{\Gamma}^*\beta \in \Omega^1(\Gamma)$.
\end{itemize}
The contact structure on $\Gamma$ coincides with the one from Theorem \ref{th:geomCE}, and the ideal Liouville structure is an $u^{-1}$-multiple of the locally confromal symplectic structure from Theorem \ref{th:geomCE}. 
\end{proposition}
\begin{proof}
We will make use of the Jacobi structure $(\Lambda,R)$ associated to $\alpha$. For the relevant concepts see Appendix \ref{sec:Jacobi}. For the points $x \in u^{-1}(\set{0})$ we have that $\alpha_x$ is in fact a smooth form. Consequently the induced Jacobi structure on $u^{-1}(\set{0})$ is simply given by restricting $\alpha$, which is indeed the contact form $\iota_{\Gamma}^*\beta$.

To describe the locally conformal symplectic structure we will use the normal forms from \cite{MO18}. By Theorem 4.1 (or Proposition \ref{prop:localform} c.f.) from \cite{MO18} there are two cases we need to consider, we will give the argument for one of these because the other is similar. The first case is when $\alpha = dx_1 + (1+y_1)\frac{dz}{z^k} + \sum_{i=2}^nx_i dy_i$. One readily computes that the associated Jacobi structure is given by:
\begin{equation*}
(\Lambda,R) = (z^k\pd{}{z}\wedge \pd{}{y_1} + \sum_{i=2}^n \pd{}{y_i}\wedge \pd{}{x_i} + (y_1+1)\pd{}{y_1}\wedge \pd{}{x_1} + \sum_{i=2}^n x_i \pd{}{x_i} \wedge \pd{}{x_1},\pd{}{x_1}).
\end{equation*}
The restriction of this Jacobi structure to $Z$ is given by
\begin{equation*}
(\Lambda_Z,R_Z) = (\sum_{i=2}^n \pd{}{y_i}\wedge \pd{}{x_i} + (y_1+1)\pd{}{y_1}\wedge \pd{}{x_1} + \sum_{i=2}^n x_i \pd{}{x_i} \wedge \pd{}{x_1},\pd{}{x_1}).
\end{equation*}
Note that this Jacobi pair is conformally equivalent to the Poisson structure
\begin{equation*}
(y_1+1)\sum_{i=2}^n \pd{}{y_i}\wedge \pd{}{x_i} + (y_1+1)^2\pd{}{y_1}\wedge \pd{}{x_1} + (y_1+1)\sum_{i=2}^n x_i \pd{}{x_i} \wedge \pd{}{x_1}.
\end{equation*}
This Poisson structure is non-degenerate and the corresponding symplectic form is given by
\begin{equation*}
(1+y_1)^{-1}\sum_{i=2}^n dx_i \wedge dy_i + (1+y_1)^{-2} \sum_{i=2}^n x_i dy_i \wedge dy_1,
\end{equation*}
which is precisely $d((1+y_1)^{-1}\beta)$.
\end{proof}
\begin{remark}\label{rem:gammanonempty}
Note that when $Z$ is compact $\Gamma$ cannot be empty, otherwise $\omega$ would be an exact symplectic form on a compact manifold.
\end{remark}

%\begin{remark}
%\alert{This is a nice construction, but unfortunately I don't see a deeper understanding.}
%Note that, in all the constructions above one has to show that they do not depend on the choice of coordinate $t$. This can be easily verified, but it can also be circumvented by using a spinor description: Consider the line $\inp{\alpha \wedge e^{d\alpha}} \subset \Lambda^{\bullet}\A_Z^*$. Taking the residue will provide a line $\inp{\Res(\alpha \wedge d^{d\alpha})} \subset \Lambda^{\bullet}T^*Z$. When restricting this line to $\Gamma$ one obtains precisely $\inp{\alpha_{\Gamma} \wedge e^{d\alpha_{\Gamma}}}$, and when restricting to the complement one obtains the line $\inp{e^{\Res(\alpha \wedge d\alpha)}}$. Moreover, $\omega =  \Res(\alpha)^{-1}\Res(\alpha \wedge d\alpha)$.
%\end{remark}

When $Z$ is co-orientable one can obtain the geometric structure on the critical set also by passing through the regularisation:
\begin{proposition}\label{prop:comparestructures}
Let $(M,Z,\alpha)$ be a co-orientable b-contact manifold, and let $(Z\times \rr,\alpha_{\reg})$ be the contact structure on the central leaf of the trivial regularisation. Then the geometric structures obtained on $Z$ through Theorem \ref{th:geomCE} and Lemma \ref{lem:hondaconvex} coincide in the following way:
\begin{itemize}
\item The submanifolds $u^{-1}(\set{0})$, is precisely the dividing set of $\alpha_{\reg}$ and the induced contact structures on these coincide
\item The locally conformally symplectic structure, $\tau \in \Omega^2(Z\backslash \Gamma)$, and the symplectic structure, $\omega \in \Omega^2(Z \backslash \Gamma)$, satisfy $\omega = u^{-1}\tau$.
\end{itemize}
\end{proposition}
\begin{proof}
Let $z$ be a defining function for $Z$, then on a neighbourhood of $Z$ we may write $\alpha = \tilde{u}d\log z + \tilde{\beta}$ with $\tilde{u} \in C^{\infty}(M)$ and $\beta \in \Gamma(p^*TZ)$. Let $(u,\beta) := (\iota^*_Zu,\iota_Z^*\beta)$, then from Lemma \ref{lem:centralleaf} it follows that $\alpha_{\reg} =uds + \beta$. Therefore, the dividing set if given by the zero-set of $f$, and on these sets $\beta$ restricts to a contact form. Comparing the formulae from above Proposition \ref{prop:bettergeom} and Lemma \ref{lem:hondaconvex} we obtain that the contact structures coincide, and we obtain the desired relation between the two-forms.
\end{proof}

The relation with convex surface theory motivates the following terminology (which we will also use in the non-coorientable case):
\begin{definition}
Let $(M^{2n+1},Z,\alpha)$ be a $b^k$-contact manifold, then we call the submanifold $\Gamma(Z):= \Res(\alpha)^{-1}(0)$ the \textbf{dividing set} of $\alpha$.
\end{definition}

One can wonder whether the ideal Liouville structure itself can be viewed as a singular symplectic form. Miranda-Oms remark that in dimension 3 the degeneracy locus does inherits a $b$-symplectic form, but this does not hold in general dimension. There is a conformal $b$-symplectic form though:
\begin{remark}\label{rem:confbsymp}
When $(F,d\lambda)$ is an ideal Liouville domain, the form $d\lambda$ is singular on $\partial F$. Because $u\lambda$ extends smoothly over $\partial F$ there exists a smooth one-form $\beta \in \Omega^1(F)$ with $\lambda = u^{-1}\beta$. Then the ideal Liouville structure is given by
\begin{equation*}
\omega = d(\beta/u) = u^{-1}d\beta + u^{-2}\beta \wedge du.
\end{equation*}
Unfortunately, this does not define a Lie algebroid form for the b-tangent bundle of $\partial F$. However, the conformal symplectic form $\tau = u d\lambda$ does define a well-defined $\A_{\partial F}$-form.
\end{remark}

%\begin{remark}[b-geometric structure on $Z$]\label{rem:logconf}
%Recall from Remark \ref{rem:confbsymp} that associated to an ideal Liouville domain we have a conformal b-symplectic structure. From the Proposition \ref{prop:comparestructures} it follows that the conformal symplectic structure on $Z\backslash \Gamma$ in fact extends to a conformal b-symplectic structure over $Z$.
%\end{remark}
%
%
%\begin{remark}[$b^k$-contact]
%Note that all the discussions work just as well for $b^k$-contact structures. \alert{ugh, shall we do this explicitly? The problem is also that Eva and Cedric also don't do it all of the time.}
%\end{remark}

The restriction of the Reeb vector field to the dividing set is compatible with the geometric structure in the following way:
\begin{lemma}\label{lem:restreebprops}
Let $(M,Z,\alpha)$ be a $b^k$-contact manifold. Then:
\begin{itemize}
\item The Reeb vector field is tangent to the dividing set $\Gamma(Z) \subset Z$. 
\item The restriction of the Reeb vector field, $R_Z$, to the complement of the dividing set is Hamiltonian for $-\Res(\alpha)^{-1}$.
\end{itemize}
\end{lemma}
\begin{proof}
On a tubular neighbourhood $p : \mathcal{U} \rightarrow Z$, let $R = g z^k\partial_z +X$, with $g \in C^{\infty}(M)$ and $X \in \Gamma(p^*TZ)$. Moreover, write $\alpha = \tilde{u} dz/z^k + \beta$, with $\tilde{u} \in C^{\infty}(M)$, and $\beta \in \Gamma(p^*T^*Z)$. Spelling out the contact condition gives
\begin{equation}\label{eq:someequations}
g\tilde{u} + \beta(X) = 1, \quad -gd\tilde{u} + X(\tilde{u})dz/z^k + \iota_Xd\beta = 0,
\end{equation}
from this it follows that $X(\tilde{u}) = 0$. Because $u = \tilde{u}|_Z$ is the divining function for the dividing set, we conclude that $X$ is tangent to $\Gamma(Z)$. Moreover
\begin{equation*}
\iota_{R_Z}(u^{-1}d\beta + u^{-2}\beta \wedge du) = u^{-1}gdu + u^{-2}(1-gu)du = u^{-2}du = d(-u^{-1}),
\end{equation*} 
showing that $R_Z$ is Hamiltonian, as desired.
\end{proof}

\subsubsection{Regularisation}

Here we will describe properties of the regularisation of $b^k$-contact structures. A similar discussion for $b^k$-symplectic structures appears in Appendix \ref{sec:Poisson}, where we will also make the connection with Poisson and Jacobi geometry. 

The foliation on the trivial regularisation is $\rr$-invariant and therefore the central leaf $Z \times \rr$ inherit an $\rr$-invariant contact structure, which we will now describe explicitly:
\begin{lemma}\label{lem:centralleaf}
Let $(M,Z)$ be a co-orientable b-pair and let $\alpha \in \Omega^1(\A_Z^k)$ be a $b^k$-contact form. Let 
\begin{equation*}
(u,\beta) := (\alpha(z\partial_z)|_Z,\iota^*_Z(\alpha - udz/z^k)).
\end{equation*}
Then the contact structure on the central leaf of any trivial regularisation $\ff$ is given by $\alpha_{\reg}|_{Z\times \rr} = uds+\beta$, with $s$ the coordinate in the $\rr$-direction.
\end{lemma}
This lemma already indicates that convex surface theory will play an important role in our study, therefore we have recalled the relevant concepts in Appendix \ref{sec:convex}.

%\begin{proof}
%If $\alpha = \tilde{u}d\log z + \tilde{\beta}$ on a neighbourhood of $Z$, then (going back to the proof of Proposition \ref{prop:trivialreg}) we have
%\begin{align*}
%\pi^*(\alpha) = \tilde{u}\frac{1}{z+\chi(z)}(dz+ds)+\tilde{\beta},
%\end{align*}
%restricting to $Z\times \rr$ yields precisely $uds+ \beta$.
%\end{proof}

We want to describe the cosymplectic structure on the singular locus of the symplectisation of a $b^k$-contact manifold. We will see that the regularisation naturally appears in this.
\begin{proposition}\label{prop:cosympofsymp}
Let $(M,Z,\alpha)$ be a co-orientable $b^k$-contact manifold, with dividing set $\Gamma \subset Z$. Let $(M\times \rr, Z \times \rr, d(e^t\alpha))$ be the $b^k$-symplectisation. Let $(\theta,\eta)$ be the induced cosymplectic (Lemma \ref{lem:logsympcosymp}) structure on $Z\times \rr$. Then: 
\begin{enumerate}
\item $e^t\theta = \theta_{\reg}$, with $\theta_{\reg} = du + udt$ the one-form defining the regularisation of $\A_{\Gamma(Z)} \rightarrow \Gamma(Z)$ via the defining function $u = \Res(\alpha)$.
\item $\rest{\eta}{\Gamma \times \rr} = d(e^t\alpha_{\Gamma})$.
\item $u\eta = \theta \wedge \beta + e^tu^2\omega$ with $\omega = u^{-1}\beta \wedge du + u^{-2}d\beta \in \Omega^2(Z\backslash \Gamma)$ the Liouville structure from Proposition \ref{prop:bettergeom}. Consequently $\eta$ and $\omega$ coincide, up to a factor, when restricted to the $(Z\backslash \Gamma)\times \rr$ leaves of the regularisation $\ff_{\reg}$ of $\A_{\Gamma(Z)}$.
\end{enumerate}
\end{proposition}
\begin{proof}
(1): Recall that $\theta = \Res(d(e^t\alpha))= d(e^t\Res(\alpha))$. Therefore $\theta = e^t(du + udt)$.\\
(2): If we write $\alpha = \tilde{u}dz/z^k + \tilde{\beta}$, then
\begin{equation*}
d(e^t\alpha) = e^tdt \wedge (\tilde{u} dz/z^k + \tilde{\beta}) + e^t(d \tilde{u} \wedge dz/z^k + d\tilde{\beta}),
\end{equation*}
and thus $\eta = e^tdt \wedge \beta + e^t d \beta$. Restricted to $\Gamma \times \rr$, this equals $d(e^t\iota^*_{\Gamma}\beta)$.\\
(3): We compute
\begin{equation*}
u\eta - e^t(du + u dt) \wedge \beta = e^t(ud\beta + \beta \wedge du) = e^tu^2\omega. \hfill \qedhere
\end{equation*}
\end{proof}
%\begin{remark}
%Recall from Remark \ref{rem:logconf} that $Z$ has a b-confromal symplectic structure $u^{-1}\omega \in \Omega^2(Z,\A_{\Gamma})$. If we lift this to a conformal symplectic foliation on $(Z\times \rr,\ff_{\reg,t})$, we obtain a symplectic foliation conformal to the one in Proposition \ref{prop:cosympofsymp}.
%\end{remark}

\subsubsection{Normal forms}

We will re-obtain normal form results for b-contact structures using normal form results in the foliated setting, so let us first recall:
\begin{proposition}[Gray Stability,\cite{CPP14}]
Let $\mathcal{F}$ be a codimension-one foliation, and $(\xi_t)_{t\in [0,1]}$ a family of codimension-two distributions such that $(\mathcal{F},\xi_t)$ is foliated contact for all $t$. Then there exists a family of diffeomoprhisms $\set{\varphi_t}_{t\in [0,1]}$ tangent to $\mathcal{F}$ with $\varphi_t^*\xi_t = \xi_0$.
\end{proposition}
From which we readily obtain a Moser statement:
\begin{lemma}\label{lem:folcontmos}[Foliated contact Moser]
Let $\xi_0,\xi_1$ be codimension-two distributions such that $(\mathcal{F},\xi_i)$ is foliated contact for $i=0,1$. Let $Z \subset M$ be a closed submanifold and assume that $(\xi_0|_{\mathcal{F}})_x = (\xi_1|_{\mathcal{F}})_x$ for all $x \in Z$. Then there exists a diffeomorphism, fixing $Z$, $\varphi : M \rightarrow M$ such that $\varphi^*\xi_1 = \xi_0$.
\end{lemma}
%\begin{proof}
%Suppose that $\xi_i = \ker \alpha_i$, and ensure that $\alpha_0$ and $\alpha_1$ coincide along $Z$. Consider the linear path $\alpha_t = (1-t)\alpha_0 + \alpha_1$. If we can show that this defines a foliated contact structure around $Z$ the result will follow by Gray stability. A straightforward computation shows that $(\alpha_t)_x \wedge (d\alpha_t)_x \wedge \theta_x > 0 $ for all $x \in Z$, where $\theta$ is a one-form defining the foliation. Therefore, there is a small neighbourhood of $Z$ on which $\alpha_t$ defines a contact foliation.
%\end{proof}

Applying this to regularisation yields:
\begin{lemma}\label{lem:regsemiglobform}
Let $(M,Z,\xi)$ be a $b^k$-contact manifold with $Z$ co-orientable, and let $(M\times \rr,\xi_{\reg,t})$ be the trivial regularisation. There is a tubular neighbourhood $p: \mathcal{U} \rightarrow Z\times \rr$ on which $\xi_{\reg,t}$ is contactomorphic to $p^*((\xi_{\reg,t})|_{Z\times \rr})$.
\end{lemma}

We can further specify this semi-local model to obtain a normal form around points in the central leaf:
\begin{proposition}
Let $x\in Z$, then there exists local coordinates $(x_i,y_i,s)$ around $(x,0) \in Z\times \rr$ such that $\xi_{\reg,t}$ has normal forms:
\begin{itemize}
\item At points in the complement of the dividing set, it is given by $\ker(ds-\sum_{i=1}^nx_idy_i) $, 
\item At points in the dividing set, it is given by $\ker(x_1ds-dy_1 - \sum_{i=2}^nx_idy_i)$.
\end{itemize}
\end{proposition}
\begin{proof}
We know that the contact form on $Z \times \rr$, $\alpha_Z$, is of the form $\alpha = uds + \beta$, with $u \in C^{\infty}(Z)$ and $\beta \in \Omega^1(Z)$. Consequently, applying a Moser argument we have that the contact foliation around a point in $Z\times \rr$ looks like $(Z\times \rr \times \rr, dz, uds + \beta)$. Then, the two different normal forms are obtained from the fact that the dividing set is given by $u = 0$.
\end{proof}
Combing the normal form in $Z\times \RR$ with Lemma \ref{lem:folcontmos}, we can find normal forms for the $b^k$-contact structure:
\begin{proposition}\label{prop:localform}
Let $(M,Z,\xi)$ be a $b^k$-contact manifold with $Z$ co-orientable. Then for $x \in Z$ we have coordinates $(t,x_i,y_i)$ around $x$ such that
\begin{itemize}
\item At points in the complement of the dividing set, $\xi$ is given by $\ker(\dfrac{dz}{z^k}-\sum_i x_idy_i)$,
\item At points in the dividing set, $\xi$ is given by $\ker(dx_1-y_1\dfrac{dz}{z^k} + \sum_{i=2}^n x_i dy_i)$.
\end{itemize}
where $Z = \set{z=0}$.
\end{proposition}
%\begin{proof}
%First we can use the Moser trick for foliated contact structures in order to assume that in the leaves close to $Z \times \RR$, the contact structure is basically a lift of $\widetilde\xi|_{Z \times \RR}$. This automatically produces local models for $\widetilde\xi$. When projected down, they yield the models above.
%\end{proof}
%\begin{remark}
%It is important to remark that although we have normal forms for the contact structure, there might be no normal forms for the associated contact forms. This is discussed in Example 4.7 \cite{MO18}. \alert{I'm not sure I also want to point out the example doesn't make sense, as they use $d\log r$ and $d\theta$. It's kinda funny though, as it really is an elliptic contact structure. Perhaps it is better to remove this remark in its entirety, although Eva would complain.}
%\end{remark}

Using the semi-local normal form around the central leaf of the regularisation we may also obtain a linear model around the degeneracy locus of a contact structure. The appropriate notion of linearity for the form is the following:
\begin{definition}[\cite{MO18}]
A $b^k$-contact form $\alpha \in \Omega^1(\A_Z^k)$ is said to be $\rr^+$\textbf{-invariant} if there exists a tubular neighbourhood of $Z$ on which it has the form
\begin{equation*}
\alpha = u\frac{dz}{z^k} + \beta, \quad f \in C^{\infty}(Z),\beta \in \Omega^1(Z). \hfill \qedhere
\end{equation*}
\end{definition}
Hence by Lemma \ref{lem:regsemiglobform} we obtain:
\begin{proposition}\label{prop:convexity}
Any $b^k$-contact structure is contactomorphic to a $b^k$-contact structure representable by an $\rr^+$-invariant form.
\end{proposition}
The above proposition was proven independently by a direct Moser argument in the recent paper \cite{CO22}.

\subsection{The elliptic setting} \label{ssec:ellipticContact}

In this section we will comment on some ways to construct examples of elliptic contact structures and their relation with b-contact structures. 

\begin{remark}[Generalized contact structures]
Elliptic symplectic structures arise from the study of special generalized complex structures. They are in one-to-one correspondence with a special class of generalized complex structures called stable.

An important question in generalized geometry is what the right notion of a ``generalized contact structure'' is. As of yet there are many competing notions (see for instance \cite{VW15} and the citations therein), however all of these have a shortcoming: Boundaries of generalized complex structures don't seem to provide examples of generalized contact structures.

Elliptic contact structures do appear as the contact boundaries of elliptic symplectic manifolds, see Lemma \ref{lem:contacthypersurface}. Given that elliptic symplectic structures correspond to stable generalized complex structures, it would be very interesting to investigate whether there is a notion of generalized contact structures which have elliptic contact structures as examples.
\end{remark}

\begin{lemma}
Let $(D^2,I_{\abs{D}} = \inp{x^2+y^2})$ be the disk with the standard elliptic divisor. Then the spaces of contact elements $(ST^*D^2)|_{D^2\backslash \set{0}}$ and $S\A^*|_{D^2\backslash \set{0}}$ are contactomorphic.
\end{lemma}
\begin{proof}
Let $\Phi : D^2\setminus \set{0} \rightarrow \rr \times S^1$ be the standard diffeomorphism, which lifts to a Lie algebroid isomorphism of $\elli \rightarrow (D^2\backslash \set{0})$ and $T(\rr \times S^1) \rightarrow \rr \times S^1$. But as $T(D^2\backslash \set{0})$ is isomorphic to $T(\rr \times S^1)$, we conclude that $\elli$ and $T(D^2\backslash \set{0})$ are isomorphic as Lie algebroids over $D^2 \setminus \set{0}$. Consequently, the spaces of contact elements are contactomorphic.
\end{proof}

% We start with understanding the spaces of contact elements in different incarnations:
%
%Consider the space of contact elements of the disk $ST^*D^2 \simeq D^2 \times S^1$. The canonical contact structure on this disk is given by $\alpha_1 = \cos(\theta_3)dx + \sin(\theta_3)dy$. 
%
%We can now also consider the space of contact elements of the annulus $ST^*(\rr_{\geq 0} \times S^1) \simeq \rr_{\geq 0} \times S^1\times S^1$, and the canonical contact structure here is given by $\alpha_2 = \cos(z)dx + \sin(z)d\theta$.
%
%We claim that these two are in fact contactomorphic. \alert{To prove}. Building on this we have the following lemma:

\begin{proposition}\label{prop:buildelli}
Let $(M^3,\xi)$ be a contact manifold, and let $L \subset M$ be a Legendrian circle. Then there exists an elliptic-contact structure with degeneracy locus $L$.
\end{proposition}
\begin{proof}
Around the Legendrian circle, the contact structure has a normal form given by the standard contact structure on $ST^*D^2$, with the fibre over zero corresponding with $L$. By the Lemma above $(ST^*D^2)\setminus D^2$ and $S\A^*\setminus D^2$ are contactomorphic, so we can perform a surgery to glue in a copy of $S\A^*$ around $L$. This gives the desired elliptic contact structure.
\end{proof}
Because any contact three-manifold has a Legendrian circle, we conclude that any contact three-manifold also admits an elliptic contact structure.

A nice interaction between b- and elliptic geometry, is given through the real oriented blow-up. This construction replaces a codimension-two co-orientable embedded submanifold $D \subset M$, with a copy of $S^1ND$, creating a manifold with boundary $p:\tilde{M} \rightarrow M$, with $\partial\tilde{M} = S^1ND$. In \cite{KL19} it was established that the real oriented blow-up of the degeneracy locus of an elliptic divisor provides a morphism of divisors $p : (\tilde{M},\partial\tilde{M}) \rightarrow (M,I_{\abs{D}})$. This morphism of divisors induces a fibre-wise isomorphism between the associated Lie algebroids, which is then used to pull-back elliptic symplectic structures on $M$ to b-symplectic structures on $\tilde{M}$. Completely similarly, we also have:
\begin{lemma}\label{lem:blowup}
Let $(M,I_{\abs{D}},\alpha)$ be an elliptic contact manifold with $D$ co-oriented, and let $p:\tilde{M} \rightarrow M$ denote the real oriented blow-up along $D$. Then $p^*\alpha$ defines a b-contact form on $(\tilde{M},\partial \tilde{M})$.
\end{lemma}

%%%%%%%%%%%%%%%%%%%%%%%%%%%%%%%%%%%%%%%%%%%%%%%%%%%%%%%%%%
\section{The Weinstein conjecture for Lie algebroid contact structures} \label{sec:Weinstein}
 
In this section we prove Lie algebroid versions of the Weinstein conjecture, under overtwistedness assumptions. The main ingredients to be used are the regularisation procedure and the foliated Weinstein conjecture:
\begin{theorem}[\cite{dPP18}]\label{th:foliatedweinstein}
Let $(M^{2n+1+m},\Fcal^{2n+1},\xi)$ be a contact foliation in a closed manifold. Let $\alpha$ be a defining one form for an extension of $\xi$ and let $R$ be its Reeb vector field. Let $\Lcal \hookrightarrow M$ be a leaf, with $\xi|_{\Lcal}$ overtwisted. Then, $R$ has a closed orbit $\gamma$ in the closure of $\Lcal$.

In general, $\gamma$ is contractible in $M$. Furthermore, if $\gamma$ is contained in $\Lcal$, it is contractible within $\Lcal$. \hfill$\triangle$
%\item If $\pi_1(\mathcal{L})\neq 0$, $R$ posses a contractible closed orbit in the closure of $\Lcal$. 
\end{theorem}

\begin{remark}[Contractibility]\label{rem:contracti}
In contrast to the classic version of this result, the orbit $\gamma$ obtained using Theorem \ref{th:foliatedweinstein} is (a priori) contractible in the ambient space $M$, but not necessarily in the leaf in which it is contained. The reason is that the proof produces a pseudoholomorphic cylinder that is graphical over $\gamma$. This cylinder is obtained from bubbling analysis on a finite energy plane. If the plane is in the leaf containing $\gamma$, it provides a nullhomotopy of $\gamma$ within the leaf. However, the plane may be in some other leaf within the closure of $\Lcal$, which proves contractibility only in the ambient.

It is still an open question to find examples of contact foliations where this is indeed the case.
\end{remark}

We will look at $b^k$-geometry first and then at elliptic geometry.

\subsection{The Weinstein conjecture in the $b^k$-contact setting}

In this subsection we present a proof of the Weinstein conjecture in overtwisted $b^k$-contact manifolds; this result appeared first in \cite{MO18}, where the developed a theory of pseudoholomorphic curves adapted to the $b^k$-structure. Our approach uses the regularistion instead, reducing the question to the foliated setting. This allows us to proof the existence of Reeb orbits in a more general setting than presented in \cite{MO18}.

Before we do so, we make some observations about the Reeb orbits in the divisor $Z$, particularly in dimensions three and five.

\subsubsection{Preliminary remarks}

\begin{lemma}[\cite{MO21}]\label{lem:cheatWeinstein}
Let $(M^{2n+1},Z,\alpha)$ be a $b^k$-contact manifold. Then:
\begin{itemize}
\item When $n= 1$ there are infinitely many periodic orbits on $Z$.
\item When $n = 2$ there is a periodic orbit on $Z$.
\end{itemize}
\end{lemma}
\begin{proof}
Because $\Gamma(Z) \subset (Z\times S^1,\alpha_{\reg})$ is a contact submanifold by Lemma \ref{lem:hondaconvex} it is immediate that it has a Reeb orbit when $n=1$ and follows from Taubes' proof of the Weinstein conjecture for $n=3$. By Lemma \ref{lem:restreebprops} we have that the Reeb vector field of $\alpha$ is tangent to $\Gamma(Z)$, and thus these orbits also define orbits of $R_{\alpha}$. 
%That these orbits are non-contractible when $n=1$ follows from the criterion of overtwistedness described in Lemma \ref{lem:Giroux}.
\end{proof}

In dimension five we can improve on this by observing the following: The Reeb vector field is tangent to all the level sets of $u= \Res(\alpha)$. Therefore, the dynamics is constrained by the structure on these level sets. But as $0$ is a regular value, for all small values $\varepsilon$ the level set $u^{-1}(\varepsilon)$ is regular as well. Moreover, the induced contact form on these manifolds is simply $\beta/\varepsilon$, and therefore they have orbits if and only if $\Gamma(Z)$ has. We arrive at the following conclusion:
\begin{lemma}\label{prop:anotheranswer}
Let $(M,Z,\alpha)$ be a $b^k$-contact manifold, and assume that the dividing set has a Reeb orbit. Then, there are infinitely many Reeb orbits on $Z$.
\end{lemma}
In dimension 5, the dividing set has dimension 3 and therefore has closed Reeb orbits. Therefore Proposition \ref{prop:anotheranswer} provides an answer to Question 6.3 in \cite{MO21}, in dimension 5. Moreover, in higher dimensions the validity of the Weinstein conjecture would give an affirmative answer in general. This fact was also observed independently in \cite{CO22}.

One can wonder whether in higher dimensions the dividing set possesses Reeb orbits. Unfortunately, by Lemma \ref{lem:tight} we have the following:
\begin{lemma}
Let $(M,Z,\alpha)$ be a b-contact manifold, then the induced contact structure on the dividing set $(\Gamma(Z),\rest{\alpha}{\Gamma(Z)})$  tight.
\end{lemma}
This observation gives a (rather unsatisfactory) ``answer'' to Question 6.10 posed in \cite{MO21}:
\begin{lemma}
A b-contact $(2n+3)$-manifold $(M,\alpha)$ without periodic Reeb orbits provides a counterexample to the ordinary Weinstein conjecture in dimension $2n+1$.
\end{lemma}
\begin{proof}
By Lemma \ref{lem:restreebprops}, the Reeb vector field is tangent to the dividing set, and thus any Reeb orbit of $(\Gamma(Z),\rest{\alpha}{\Gamma(Z)}$ will be a Reeb orbit of $(M,\alpha)$. Thus if $M$ has no periodic Reeb orbits, than in particular the tight contact manifold $(\Gamma(Z),\rest{\alpha}{\Gamma(Z)})$ should have no periodic Reeb orbits.
\end{proof}

\subsubsection{Weinstein conjecture for overtwisted $b^k$-contact forms}

To obtain Reeb orbits we will apply the following lemma:
\begin{lemma}\label{lem:projectingReeb}
Let $(M,Z,\alpha)$ be a $b^k$-contact manifold, and let $(M\times S^1,\ff_{\reg,c},\alpha_{\reg})$ be its compact regularisation. If $\gamma$ is an orbit of the Reeb vector field of $\alpha_{\reg}$, then $\pi(\gamma)$ is a Reeb orbit of $R_{\alpha}$.
\end{lemma}
\begin{proof}
If $\gamma$ is a Reeb orbit away from the central leaf $Z\times S^1$, it lies in one of the contact leaves graphical over $M\backslash Z$. As the projection is a contactomorphism of the leaf with $M\backslash Z$ it is immediate that $\pi(\gamma)$ is a Reeb orbit of $\alpha$. 

Suppose that $\gamma$ is a Reeb orbit in the central leaf. Around $Z$ we may write $R_{\alpha} = \tilde{g}z\partial_z + \tilde{X}$. Then $R_{\alpha_{\reg}}|_{Z\times S^1} = g\partial_{\theta} + X$, where $g,X$ are the restrictions of $\tilde{g}$ respectively $\tilde{X}$. We see that $R_{\alpha_{\reg}}|_{Z\times S^1}$ projects to $Z$ as $X$, as this is precisely the restriction of $R_{\alpha}$ to $Z$ we obtain the desired conclusion.
\end{proof}

Using the foliated Weinstein conjecture we can re-obtain a Weinstein conjecture for overtwisted $b^k$-contact forms:
\begin{theorem}\label{th:Weinnormal}
Let $(M^{2n+1},Z,\alpha)$ be a $b^k$-contact manifold. Assume that there exists an overtwisted disk in $M \backslash Z$ or in the central leaf $Z \times \SS^1$ of the regularisation. Then, either:
\begin{itemize}
\item There is a contractible Reeb orbit of $\alpha$ in $M \backslash Z$.
\item There is a Reeb orbit in $Z\times \SS^1$, which corresponds to a non-constant Reeb orbit of $\alpha$ on $Z$.
\end{itemize}
\end{theorem} 
\begin{proof}
If $Z\times \SS^1$ is overtwisted, it admits a contractible Reeb orbit according to Hofer's theorem. We then work under the assumption that $M\backslash Z$ is overtwisted.

Let $(M \times \SS^1, \Fcal, \tilde \xi)$ be the compact regularisation of $(M^{2n+1},Z,\alpha)$. The leaves of $\Fcal$ different from the central leaf $Z \times \SS^1$ are graphical over $M \backslash Z$. Consequently, if the contact manifold $M\backslash Z$ has an overtwisted disk, each of these leaves is an overtwisted contact manifold. Fix such a leaf $L$ and apply Theorem \ref{th:foliatedweinstein} to obtain a contractible periodic Reeb orbit $\gamma$ in its closure. The closure of $L$ consists of the leaf itself, together with the central leaf $Z \times \SS^1$.

If the orbit appears in $L$ itself, we immediately obtain an orbit in $(M\backslash Z,\alpha)$. It is contractible because it bounds a finite energy plane arising from bubbling. If the Reeb orbit $\gamma$ appears instead in the central leaf, we must argue differently. According to Theorem \ref{th:foliatedweinstein}, $\gamma$ is contractible in the ambient manifold $M \times \SS^1$. In particular, its projection to the $\SS^1$-factor in $Z \times \SS^1$ is contractible. It follows that $\gamma$ is not a vertical orbit, meaning that its projection to $Z$ is non-constant. By Lemma \ref{lem:projectingReeb}, $\pi(\gamma)$ is the desired Reeb orbit of $\alpha$ on $Z$.
\end{proof}

%\begin{remark}
%When $(M^3,Z,\alpha)$ is an overtwisted contact manifold with $(Z\times S^1,\alpha_{\reg})$ overtwisted it is immediate from Lemma \ref{lem:cheatWeinstein} that there are \emph{contractible} orbits on $Z$. However when $(Z\times S^1,\alpha_{\reg})$ is tight this is no longer the case. It would be nice to have some examples of contact manifolds which are overtwisted in $M \setminus Z$ but tight in $Z\times S^1$. \alert{We don't have any yet, if I remember correctly.}
%\end{remark}

\subsubsection{The $\RR^+$-invariant setting}

Theorem \ref{th:Weinnormal} was proven in \cite{MO18} under the simplifying assumption that the $b^k$-contact manifold is $\rr^+$-invariant. However, in this case more can be obtained as orbits will always exist in $M\backslash Z$. We will recover their result as follows.

First Lemma \ref{lem:centralleaf} gives an explicit description of the contact structure on the central leaf. When the contact form is $\rr^+$-invariant this is precisely the model for $M\backslash Z$:
\begin{lemma}\label{lem:normalforminv}
Let $(M,Z,\alpha)$ be an $\rr^+$-invariant $b^k$-contact form, with $Z$ co-orientable. There exists a tubular neighbourhood $\mathcal{U}$ of $Z$, such that:
\begin{align*}
\varphi : (Z \times \rr,\alpha_{\reg,t}) &\rightarrow (\U\backslash Z,\alpha)\\
(x,s) &\mapsto (x,e^s)
\end{align*}
is a strict contactomorphism onto its image when $k = 1$, and
\begin{align*}
\varphi : (Z \times \rr_{>0},\alpha_{\reg,t}) &\rightarrow (\U\backslash Z,\alpha)\\
(x,s) &\mapsto (x,\lambda_k s^{\frac{1}{1-k}}),
\end{align*}
when $k \geq 2$ and $\lambda_k$ is some constant depending on $k$.
\end{lemma}
Consequently, we see that if $(Z \times \rr_{>0},\alpha_{\reg,t})$ has a closed Reeb orbit, so does $(M\backslash Z,\alpha)$. However, in Theorem \ref{th:Weinnormal} the orbits we obtained were on the compact regularisation, not on the trivial one. We will now show that these orbits do however always corresponds to closed orbits in in the trivial regularisation, by which we re-establish the Weinstein conjecture for $\rr^+$-invariant $b^k$ contact structures:

\begin{theorem}\label{th:Weininvariant}
Let $(M^{2n+1},Z,\alpha)$ be a closed overtwisted $b^k$-contact manifold and assume that there is a neighbourhood of $Z$ on which $\alpha$ is $\rr^+$-invariant. Then there exist non-constant contractible Reeb orbits in $M\setminus Z$.
\end{theorem}
\begin{proof}
By Theorem \ref{th:Weinnormal} there exists a non-constant Reeb orbit $\gamma$ either in $M\backslash Z$ or in central leaf $(Z\times S^1,\alpha_{\reg,c})$. In the first case we immediately get the desired orbit of $\alpha$ in $M \backslash Z$, so let us focus on the second case. We may consider $(Z \times \rr,\alpha_{\reg,t}) \rightarrow (Z\times S^1,\alpha_{\reg,c})$ as a covering space, for which the quotient map is a strict contactomorphism. Because the projection of $\gamma$ onto $S^1$ is contractible, we have that the lift of $\gamma$ to $Z \times \rr$ is closed. Because the contact structure is $\rr$-invariant this orbit can be translated so that it lies within $Z \times \rr_{>0}$. Using Lemma \ref{lem:normalforminv} we obtain the desired orbit in $M\setminus Z$.
\end{proof}

\subsubsection{``True'' orbits in the divisor}

In the above theorems we showed that certain orbits in $Z$ can be ``pushed out'' by passing through $(Z \times \rr,\alpha_{\reg,t})$ and applying Lemma \ref{lem:normalforminv}. The first thing these orbits have to satisfy is that they remain closed in $Z \times \rr$, that is why we consider the following definition:
\begin{definition}
Let $(M,Z,\alpha)$ be a $b^k$-contact manifold, and let $\gamma$ in $Z$ be a closed Reeb orbit. We say that $\gamma$ is a \textbf{true} orbit if the points in $\gamma$ lie in a closed orbit in $(Z\times \rr,\alpha_{\reg,t})$. 
\end{definition}

\begin{remark}
Let $(M,Z,\alpha)$ be an $\rr^+$-invariant $b^k$-contact manifold. Then $\alpha = u z^{-k} dz + \beta$ on a tubular neighbourhood, and $\alpha_{\reg,Z} = uds + \beta$. If we let If we let $R = gz^k\partial_z + Y$ denote the Reeb vector field of $\alpha$, with $g \in C^{\infty}(Z)$ and $Y \in \mathfrak{X}^1(Z)$, then $R_{\reg,Z} = g\partial_s + Y$.  
%In case $\gamma$ is an orbit in $Z$ along which $g$ does not vanish, then the orbit does not lift to a closed orbit in $Z\times \rr$ as there is always a non-zero $\partial_s$-component.  \alert{Because of the quotient we're taking doesn't every orbit with a bit of vertical component near $Z$ wrap around $S^1$?}\\\\
In case $g\equiv 0$ on a Reeb orbit $\gamma$ of $R$, then the orbit $\gamma \times \set{c}$ is an orbit or $R_{\reg,Z}$ for every constant $c$. However, if $g$ is non-zero on a Reeb orbit of $R$, then it could very well be the case that the corresponding orbit of $R_{\reg,Z}$ looks very different. In particular, the corresponding Reeb orbit will open up. 
\end{remark}
%Because the Reeb vector field $R_{\reg,Z}$ is tangent to the dividing set $\Gamma(Z)$, we see that any Reeb-orbits contained in $\Gamma(Z)$ automatically lift to orbits in $Z \times \rr$. We conclude (\alert{I remember I had my doubts about this one, yes it cannot be correct but hmmm}:
%\begin{lemma}
%Let $(M,Z,\alpha)$ be an $\rr^+$-invariant b-contact manifold. If there exists a Reeb orbit in one of the contact leaves in $Z$, then there exists a Reeb orbit in $M \backslash Z$. 
%\end{lemma}
%Combining this with the existence of Reeb orbits on 1 and 3-dimensional contact manifolds we immediately conclude that there always exists orbits in the complement of any three or five-dimensional $\rr^+$-invariant contact manifold. 
%\end{remark}

%% TODO: remark that the example of Eva of the sphere has no true orbits either
The following example is discussed in \cite[Example 6.8]{MO21}, we comment on it because it neatly visualises the construction of orbits:
\begin{example}
Consider the $\rr^+$-invariant b-contact manifold $(T^3,\alpha = \sin \theta_2 \frac{d\theta_1}{\sin \theta_1} + \cos \theta_2 d\theta_3)$ with Reeb vector field $R_{\alpha}=\sin \theta_2 \sin \theta_1 \partial_{\theta_1} + \cos \theta_2 \partial_{\theta_3}$. The restriction of $R_{\alpha}$ to $Z = \set{\theta_1 = 0}$ is given by $\cos \theta_2 \partial_{\theta_3}$. Therefore the sets $\set{\theta_1 = 0,\pi, \quad \theta_2 = \text{const}.}$ define orbits, which are non-constant if $\theta_2 \neq \pi/2,3\pi/2$. 
%The vector field $\sin\theta_1 \partial_{\theta_1}$ is a strict contact vector field, and is transverse to the boundary of a tubular neighbourhood of $Z$, consequently $\alpha$ is $\rr^+$-invariant.
Note that $(T^3\backslash Z)$ is contactomorphic to $(Z\times \rr^*,\alpha_{\reg} = \sin \theta_2 ds + \cos \theta_2 d\theta_3)$. The dividing set of $Z$ if given by $\set{\theta_2 = 0} \cup \set{\theta_2 = \pi}$, and thus by Giroux' criterion $Z \times \rr^*$ is tight, and so is $T^3/Z$.

Therefore, we cannot apply Theorem \ref{th:Weininvariant} to deduce that their are closed orbits in $T^3/Z$, still we can explicitly find the orbits. We will do this by considering how the orbits on $Z$ correspond to orbits on the trivial regularisation $Z \times \rr$. The regularised Reeb vector field restricted to $Z\times \rr$ is given by $R_{\reg,t,Z} = \sin \theta_2 \partial_s + \cos \theta_2 \partial_{\theta_3}$. The induced orbits on $Z \times \rr$ come in three flavours:
\begin{itemize}
\item Vertical : When $\theta_2 = \pi/2,3\pi/2$ the orbits are of the form $\gamma(t) = (y,\theta_2,t)$ and correspond to fixed points of $R_{\alpha}$. Note that the induced orbits of the compact regularisation would be closed (wrapping around the $S^1$-direction).
\item Slanted, and wrapping around: When $\theta_2 \neq 0,\pi/2,\pi,3\pi/2$ the Reeb vector field has nowhere vanishing $\partial_s$ component and so the orbits are slanted and go to infinity in both directions.
\item Horizontal: When $\theta_2 = 0,\pi$ the orbits $\gamma(t) = (t,\theta_2,c)$ with the first coordinate in the y-direction are horizontal and closed.
\end{itemize}
Therefore, by Lemma \ref{lem:normalforminv} these last orbits induce closed orbits in $M\backslash Z$.  Explicitly these are given by the orbits of $\partial_{\theta_3}$ when $\theta_2 = k\pi$ (These orbits were overlooked in \cite{MO21}.)
\end{example}
%\alert{The above example is nice, but it is very special. The Reeb orbits on the dividing sets turn out to be true orbits, because the $\partial_s$ component of the regularised Reeb vanishes precisely on the dividing set. So, it would be good to have an example where this does not happen (besides the one CE give), so that we can show that some of the ``cheated'' orbits are fake.}
%
%
%\alert{Putting in some more examples later wouldn't hurt.}
%
%\alert{It would be nice to have a compact $\rr^+$-invariant, tight b-contact manifold with only fake orbits on $Z$ and no orbits in the complement. This is kinda what CE claim in the above example, but that was false.}

\subsection{The Weinstein conjecture for elliptic contact forms}

We would also like to use the regularisation procedure to produce Reeb orbits in an elliptic contact manifold $(M,\alpha)$. Even though we can apply the foliated Weinstein conjecture to obtain orbits in the regularisation, unfortunately we can  not prove that they yield orbits on $M$. For b-contact manifolds we were able to use contractibility of the orbits to exclude the case that they wrap around vertical in $Z\times \SS^1$. This reasoning does not work in the elliptic setting.

However, we can prove the Weinstein conjecture for elliptic contact manifolds as long as some invariance is present. Following the notion of $\rr^+$-invariance for $b^k$-contact forms, we can define the equivalent notion for elliptic contact:
\begin{definition}
An elliptic contact form $\alpha \in \Omega^1(\elli)$ is said to be $\cc^*$-\textbf{invariant} if there exists a tubular neighbourhood of $D$ in which it has the form
\begin{equation*}
\alpha = u_1 d\log r + u_2 d\theta + \beta, \quad, u_1,u_2 \in C^{\infty}(D), \beta \in \Omega^1(D). \hfill \qedhere
\end{equation*} 
\end{definition}

\begin{remark}
Using an identical reasoning as in Proposition \ref{prop:convexity} one can show that any co-orientable elliptic contact structure can be represented by a $\cc^*$-invariant elliptic contact form. Natural examples of $\cc^*$-invariant contact forms are obtained through Proposition \ref{prop:buildelli}.
\end{remark}

\begin{theorem}
Let $(M^{2n+1},I_{\abs{D}},\alpha)$ be a $\cc^*$-invariant elliptic contact manifold. Then there exists a closed Reeb orbit in $M\backslash D$.
\end{theorem}
\begin{proof}
We consider the real oriented blow-up along $Z$, and the b-contact structure it inherits via Lemma \ref{lem:blowup}: $(\tilde{M},\partial\tilde{M},p^*{\alpha})$. Note that $p^*\alpha$ is an $\rr_+$-invariant contact structure and thus we apply Theorem \ref{th:Weinnormal} to obtain the existence of a Reeb orbit in $\tilde{M}\backslash \partial \tilde{M}$. Strictly speaking Theorem \ref{th:foliatedweinstein} (and consequently Theorem \ref{th:Weinnormal}) is only proven for manifolds without boundary. However, note that the compact regularisation on $\tilde{M}\times S^1$ is well-behaved: All leaves have no boundary, and the leaf $(\partial \tilde{M})\times S^1$ is precisely the boundary of $\tilde{M} \times S^1$. Therefore, it is clear that we may still apply Theorem \ref{th:foliatedweinstein} and Theorem \ref{th:Weinnormal} to obtain a closed Reeb orbit in $\tilde{M}\backslash \partial \tilde{M}$. Because $p: \tilde{M}\backslash p^{-1}(D) \rightarrow M\backslash D$ is a contactomorphism, we conclude the existence of the desired orbits in $M \backslash D$. 
\end{proof}

\subsection{The Weinstein conjecture in self-crossing $b$-geometry}

We can also straightforwardly adapt the reasoning to b-contact forms with immersed hypersurface:
\begin{theorem}
Let $(M,Z)$ be a manifold with an immersed hypersurface as in Definition \ref{def:logx}. Let $\alpha \in \Omega^1(\A_Z)$ be a $b$-contact from and assume that there is an overtwisted disk in $M\backslash D$. Then, the Reeb vector field has non-constant closed orbits.
\end{theorem}
\begin{proof}
Let $n$ be the intersection number of the b-divisor. The regularisation (Proposition \ref{prop:reglogx}) provides us with a contact foliation $(\ff,\alpha)$ on $M \times T^n$, and using the overtwisted disk in one of the $(M\setminus Z)$ leaves we may apply Theorem \ref{th:foliatedweinstein} to obtain a Reeb orbit $\gamma$ on $M \times T^n$.

Let $Z = Z_1 \cup \ldots \cup Z_n$, then the leaves of $\ff$ are given by 
$L_1~ {}_{\pi_1} \times_{\pi_2} L_2 ~{}_{\pi_2}\times_{\pi_3} \cdots {}_{\pi_{k-1}}\times_{\pi_k} L_k$, with $L_i$ a leaf of the compact regularisation of $\A_{Z_i}$. As in the case for a smooth divisor, we have to argue that the projection of $\gamma$ to $M$ is non-constant. Suppose, without loss of generality, that $L_1,\ldots,L_l$ are leaves graphical over $M\backslash Z_1,\ldots,M\backslash Z_l$ respectively, then $L_i = Z_i \times S^1$ (or a connected component thereof) for the remaining $l< i \leq n$. Because $\gamma$ is contractible in $M \times T^n$, the projection of $\gamma$ to each of these $S^1$-factors has to be contractible. It follows that $\gamma$ is not vertical, and hence the projection to $Z$ is non-constant.
\end{proof}
Moreover, if one assumes that a self-crossing elliptic contact form is invariant around each of the hypersurfaces, one can blow-up all hypersurfaces an obtain an invariant self-crossing b-contact structure and obtain Reeb orbits on $M\backslash D$.

%Analogously we can obtain use the regularisation (Proposition \ref{prop:regellx}) to obtain a result for self-crossing elliptic structures:
%\begin{theorem}
%Let $(M,I_{\abs{D}})$ be a manifold with a self-crossing elliptic divisor as in Definition \ref{def:ellx}. Let $\alpha \in \Omega^1(\elli)$ be an overtwisted contact form. Then the Reeb vector field has non-constant closed orbits.
%\end{theorem}

%------------------------------------------------------------------------
\section{Other Lie algebroid distributions} \label{sec:otherDistributions}

We now present various constructions of Lie algebroid distributions outside of the contact setting. Namely, we introduce the annihilator (Subsection \ref{ssec:annihilatorDistribution}) and sphere annihilator constructions (Subsection \ref{ssec:sphereAnnihilatorDistribution}), which generalise the symplectisation and the space of contact elements, respectively.

\subsection{Even-contact structures and Engel structures}

Before we discuss constructions, let us introduce a couple of relevant families of distributions. For algebroids of even rank, one can define the following analogue of contact structures:
\begin{definition}
Let $\A \to M$ be a Lie algebroid of rank $2r$. A corank-$1$ distribution $\xi \subset \A$ is said to be an \textbf{$\A$-even-contact structure} if its first curvature
\[ \Omega_{1,1}(\xi): \xi \times \xi \longrightarrow \A/\xi \]
has $1$-dimensional kernel.
\end{definition}

In analogy with the classic case, another interesting class of structures is:
\begin{definition}
Let $\A \rightarrow M$ be a Lie algebroid of rank $4$. An $\A$-\textbf{Engel structure}, is a regular distribution $\xi \subset \A$ of rank $2$ with $[\xi,\xi] \subset \A$ even-contact.
\end{definition}

\begin{example}
Let $\A = \mathcal{H}_r \oplus \rr$, the sum of the Heisenberg Lie algebra and a trivial factor. Then $\xi$ from Example \ref{ex:Heiscont} defines an even-contact structure.

Let $\A$ be the 4-dimensional Lie algebra with basis $\{e_1,e_2,e_3,e_4\}$ with $[e_1,e_2] = e_3$ and $[e_2,e_3] = e_4$ and all other brackets zero. Then $\xi = \text{span}(e_1,e_2)$ is Engel.
\end{example}

\subsection{The annihilator of a distribution} \label{ssec:annihilatorDistribution}

%% TODO: talk about flatness in terms of the Bott connection
In the classical setting, the annihilator of a distribution can be equipped with the restriction of the Liouville form. It turns out that the non-involutivity of the distribution is then encoded in the geometry of this restricted form. We now explain how this adapts to general Lie algebroids.

Let $\A \rightarrow M$ be a Lie algebroid and write $\pi: \A^* \to M$ for its dual. Given a distribution $\xi \subset \A$, we write 
\[ \xi^\bot := \{ \alpha \in \A^* \,\mid\, \alpha|_\xi = 0 \} \]
for its \textbf{annihilator}. This is a submanifold of $\A^*$ (and in particular a vector subbundle). For simplicity, let us write $p: \xi^\bot \to M$ for the restriction of $\pi$ to $\xi^\bot$.

The pullback algebroid $p^!\A \rightarrow \xi^\bot$ sits naturally inside $\pi^!\A|_{\xi^\bot}$ and, as such, we can restrict $\lambda_\can$ to it. We then define:
\begin{definition}
Let $\A \rightarrow M$ be a Lie algebroid. The \textbf{Bott form} of the distribution $\xi \subset A$ is:
\[ \omega^\xi_\can := \omega_{\can}|_{p^!\A}. \]
Similarly, we denote $\lambda^\xi_\can := \lambda_{\can}|_{p^!\A}$.
\end{definition}
We call it the Bott form because it defines the Bott connection of $\xi$ if $\xi$ is involutive; see Lemma \ref{lem:Bott} below.
%% TODO: add Bott connection

In local terms, if $(\alpha_1,\ldots,\alpha_r)$ is a coframe of $\A^*$ such that $(\alpha_1, \ldots,\alpha_k)$ spans $\xi^\bot$, we have that:
\begin{equation*}
\lambda^\xi_{\can} = t_1\alpha_1 + \cdots + t_k \alpha_k,
\end{equation*}
where $(t_1,\ldots,t_r)$ are the associated fibrewise coordinates in $\A^*$ (so, in particular, $(t_1,\ldots,t_k)$ are fibrewise coordinates in $\xi^\bot$). Similarly, the Bott form reads:
\begin{equation*}
\omega_{\can}^\xi = \sum_{i=1}^k dt_i \wedge \alpha_i + t_i d\alpha_i,
\end{equation*}
which fails to be symplectic along the zero section unless $\xi = 0$.

The more interesting question is whether $\omega^\xi_{\can}$ is symplectic in $\xi^\perp \setminus M$. As in the classical case:
\begin{lemma} \label{lem:Bottcontact}
Let $\xi \subset \A$ be an $\A$-distribution of corank $1$. Then, $\xi$ is contact if and only if its Bott form is symplectic away from the zero section.

Furthermore, if $\xi = \ker(\alpha)$, with $\alpha \in \Gamma(\A^*)$, then $(p^!\A|_{\xi^\perp \setminus M},\omega_\can^\xi)$ has two components and both are isomorphic to the symplectisation of $(\A \to M,\alpha)$.
\end{lemma}

At the opposite end of the spectrum:
\begin{lemma} \label{lem:Bott}
Let $\xi \subset \A$ be an $\A$-distribution of rank $k$. Then, $\xi$ is involutive if and only if its Bott form is of constant rank $2k$.

If this is the case, the kernel $\ker(\omega_\can^\xi) \subset  p^!\A$ defines a flat connection in $\xi^\perp$, which we call the \textbf{Bott connection}.
\end{lemma}
%% TODO: compare to Bott connection of Madeleine. Don't wanna, here papers take so long to read...

\subsection{The sphere annihilator of a distribution} \label{ssec:sphereAnnihilatorDistribution}

Consider now the annihilator sphere bundle $p_\SS: \SS\xi^\perp \to M$. Since $\RR^*$ acts by $p^!\A$-algebroid morphisms we obtain a Lie algebroid $p^!_\SS\A$ on $\SS\xi^\perp$. Due to homogeneity of $\lambda^\xi_\can$, we similarly obtain a corank-$1$ distribution $\Dcal_\can^\xi \subset p^!_\SS\A$.

The geometric interpretation of $\Dcal_\can^\xi$ is that it is the universal object encoding hyperplane $\A$-distributions containing $\xi$.

In analogy with Lemma \ref{lem:Bott}, it can be checked that $\xi$ is involutive if and only if $\Dcal_\can^\xi$ is involutive. In the opposite direction, a case that is of interest (at least in the classical setting) is the following:
\begin{definition}
An $\A$-distribution $\xi$ is said to be \textbf{fat} if $\Dcal_\can^\xi$ is contact.
\end{definition}
Equivalently, combining Lemmas \ref{lem:contactElements} and \ref{lem:Bottcontact}, yields that $\xi$ is fat if and only if its Bott form is symplectic.

\begin{example}
Let $\frakg$ be the $6$-dimensional Lie algebra defined by the brackets $[e_1,e_2] = [e_3,e_4] = e_5$ and $[e_1,e_3] = -[e_2,e_4] = e_6$ and consider the distribution $\xi = \langle e_1,e_2,e_3,e_4 \rangle$. This algebra may be endowed with a compatible grading, in which $\xi$ is the degree-$1$ piece and $\langle e_5, e_6 \rangle$ is the degree-$2$ piece. With this grading, $\frakg$ is said to be the \textbf{elliptic} algebra of rank 4 and dimension 6.

We then consider the annihilator $\xi^\bot \subset \frakg^*$, and the pullback Lie algebroid $\xi^\bot \oplus \frakg$ over it. If $\{e_i^*\}$ denotes the basis dual to $\{e_i\}$, $\xi^\bot$ is spanned by $e_5^*$ and $e_6^*$; we write $(a,b)$ for the corresponding coordinates. Then, the Bott form reads:
\begin{align*}
\lambda_\can^\xi & = ae_5^* + be_6^* \\
\omega_\can^\xi  & = (da \wedge e_5^* + db \wedge e_6^*) + a de_5^* + bde_6^* \\
							   & = (da \wedge e_5^* + db \wedge e_6^*) + a(e_1^* \wedge e_2^* + e_3^* \wedge e_4^*) + b(e_1^* \wedge e_3^* - e_2^* \wedge e_4^*),
\end{align*}
which is readily seen to be symplectic away from $0 \in \xi^\bot$, showing that $\xi$ is fat.
\end{example}

\subsection{Prolongation of distributions} \label{ssec:prolongation}

Dually to the previous subsection, we now define the universal object encoding all the lines contained in an $\A$-distribution.
\begin{definition}
Let $\A \rightarrow M$ be a Lie algebroid and $\xi$ be an $\A$-distribution. Its \textbf{(Cartan) prolongation} is the pair 
\[ (\pi^!\A \rightarrow \PP\xi, \DD_\can), \]
where $\pi: \PP\xi \to M$ is the fibrewise projectivisation of $\xi$ and
\[ \DD_\can(L) := \{ v \in \pi^!\A \,\mid\, v \in L \}, \]
where $L$ is an element of $\PP\xi$ (i.e. a line in $\xi$).
\end{definition}
%\footnote{Instead of considering the space of lines in $\xi$, one may want to consider higher grassmannians of $\xi$. Although possible, it turns out that this is the wrong generalisation to higher dimensional subspaces. Instead, one should consider the grassmannian consisting of those subspaces of $\xi$ that are isotropic with respect to the curvature. The issue with this notion is that, in general, such a grassmannian is a potentially singular algebraic variety. We leave it to the reader to explore this further in the Lie algebroid setting. \alert{I don't know about such a remark.}}

The following Lemma provides many examples of algebroid Engel structures:
\begin{lemma}
Let $\A \rightarrow M$ be a Lie algebroid of rank $3$. Let $\xi \subset \A$ be an $\A$-contact structure with projectivisation $\pi:\PP\xi \to M$. Then, the prolongation $\DD_\can$ of $\xi \to \PP\xi$ is a $\pi^!\A$-Engel structure.
\end{lemma}
The projectivisation procedure can then be iterated (applying it to $\DD_\can)$ to produce further examples of Lie algebroid bracket-generating distributions of rank $2$.

\appendix
%%%%%%%%%%%%%%%%%%%%%%%%%%%%%%%%%%%%%%%%%%%%%%%%%%%%%%%%%%%%%%%%%%%%%%%%%%%%%%

\section{Regularisation from a Poisson and Jacobi perspective}\label{sec:Poisson}

In this appendix we will describe how the regularisation procedure interacts with Poisson and Jacobi geometry. In this paper we mostly focused on regularising $b^k$-contact manifolds, but of course many more geometric structures can be regularised. For $\A_Z$-symplectic structures we have: 
\begin{proposition}\label{prop:reglogsymp}
Let $(M,Z,\omega)$ be a b-symplectic manifold. Then, $(M\times S^1,\Fcal_{\reg,c},(d\pi)^*\omega)$ defines a codimension-one symplectic foliation. 
%It can be made strong in a neighbourhood of $Z \times \RR$
\end{proposition}
\begin{proof}
Because $\pi$ is a Lie algebroid morphism, it is immediate that $(d\pi)^*\omega$ defines a closed foliated form, so we are left to show that it is leafwise non-degenerate. Because $d\pi$ is a surjective submersion between vector bundles of the same rank it is a fibre-wise isomorphism. From this it immediately follows that $(d\pi)^*\omega$ is non-degenerate.
%This is obvious away from a tubular neighbourhood of $Z\times \RR$. On $\Ucal \times \RR$ we have $\alpha = \chi(t)dt + tds$, and hence $\omega^n \wedge \alpha = t\omega^n \wedge ds$. Because $\omega$ is b-symplectic, $t\omega^n$ is a volume form on $M$ so we find that $\omega$ is leafwise non-degenerate. Moreover, because $d\pi$ is an algebroid morphism, $(d\pi)^*$ is a cochain map and therefore $(d\pi)^*\omega$ is closed along the leaves. In local coordinates we may write $\omega = d\log t \wedge \gamma +\beta$. For this to be globally closed, we need $(d\pi)^*(d\log t)$ to be closed. We compute:
%\begin{align*}
%(d\pi)^*(d\log t) &= \frac{1}{t+\chi(t)}(dt + ds),\\
%d(d\pi)^*(d\log t) &= -\frac{1+\chi '(t)}{(t+\chi(t))^2}ds\wedge dt.
%\end{align*}
%So only if we ensure that $\chi(t) = 1-t$ we get that $(d\pi)^*(d\log t)$ is closed. (I find this particular form of $\chi$ a bit strange, did I make a mistake?)
\end{proof}
Interestingly enough this result has already appeared in the literature under several different guises. We will elaborate on this in the following pages.
\subsection{Poisson}

We recall that b-symplectic structures can be viewed as Poisson structures which degenerate along a hypersurface:
\begin{lemma}\cite{GMP14}
Let $(M^{2n},Z^{2n-1})$ be a b-pair. There is a one-to-one correspondence between:
\begin{itemize}
\item b-symplectic structures $\omega \in \Omega^2(\A_Z)$.
\item Poisson structures $\pi \in \mathfrak{X}^2(M)$ for which $\wedge^n\pi$ vanishes transversely along $Z$.
\end{itemize}
\end{lemma}

In \cite{OT15} a construction is presented of how to construct out of an orientable b-symplectic manifold a codimension-one symplectic foliation:
\begin{proposition}[Theorem 4.2.2 in \cite{OT15}]\label{prop:boris}
Let $(M,\pi)$ be an orientable b-symplectic manifold and let $Z$ be the singular locus of $\pi$. Let $X$ be a Poisson vector field that is transverse to the symplectic foliation on $Z$. Then $\pi_R = \pi + X \wedge \partial_{\varphi}$ is a regular corank-one Poisson structure on $M\times S^1$
\end{proposition}
If we let $\abs{z} : M \rightarrow \rr_{\geq 0}$ be the distance to $Z$ with respect to some metric, then the symplectic foliation as in Proposition \ref{prop:reglogsymp}, coincides with the above for $X = \pi^{\sharp}(d\log \abs{z})$. Thus concluding that the construction of Osorno-Torres is a special case of the regularisation.

In fact, it turns out that Proposition \ref{prop:boris} is a special case of a more general construction for Poisson structures on line bundles.
We consider Poisson structures $\pi \in \mathfrak{X}^2(E)$ on line bundles $E\rightarrow (M,\pi)$ which are \textbf{$\rr^*$-invariant}, that is $m_{\lambda}^*\pi_E = \pi_E$ for all $\lambda \in \rr$, where $m_{\lambda} : E \rightarrow E$ denotes the multiplication by scalars. These Poisson structures can be related with the following notion:
\begin{definition}
Let $(M,\pi)$ be a Poisson manifold, and $E \rightarrow M$ a vector bundle. A \textbf{contravariant connection} is a $(T^*M,[\cdot,\cdot]_{\pi})$-connection, or, explicitly a map
\begin{align*}
\nabla^{\pi} : \Gamma(T^*M) \times \Gamma(E) \rightarrow \Gamma(E),\\
(\alpha,s) \mapsto \nabla^{\pi}_{\alpha}(s),
\end{align*}
which is $C^{\infty}(M)$-linear in the first entry and satisfies the following Leibniz identity in the second:
\begin{equation*}
\nabla^{\pi}_\alpha(fs) = f\nabla^{\pi}_{\alpha}(s) + \pi(df,\alpha)s.
\end{equation*}
Its \textbf{curvature} is the endomorphism valued bivector field $K^{\nabla^{\pi}} \in \mathfrak{X}^2(M;\End(E))$, defined by
\begin{equation*}
K^{\nabla^{\pi}}(\alpha,\beta)(s) = \nabla^{\pi}_{\alpha}(\nabla^{\pi}_{\beta}(s)) - \nabla^{\pi}_{\beta}(\nabla^{\pi}_{\alpha}(s)) - \nabla^{\pi}_{[\alpha,\beta]_{\pi}}(s).
\end{equation*}
A contravariant connection is called \textbf{flat} if its curvature vanishes.
\end{definition}

Contravariant connections induce Poisson structures on the total space of the line bundle:
\begin{proposition}[\cite{Pol97}]\label{prop:poli}
Given a line bundle $E \rightarrow (M,\pi)$. There is a one-to-one correspondence between $\rr^*$-invariant Poisson structures on $E^*$ and flat contravariant connections on $E$. 
\end{proposition}
\begin{proof}
Because $E$ has rank one, to define a Poisson bracket on $E^*$ it suffices to specify it on pairs of fibre-wise constant functions and pairs of fibre-wise constant and fibre-wise linear functions. Using that fibre-wise linear functions on $E^*$ are in one-to-one correspondence with sections $\Gamma(E)$ we can thus define
\begin{equation}\label{eq:expbracket}
\set{p^*f,p^*g}_{E^*} := p^*(\set{f,g}_M),\quad \set{p^*f,s}_{E^*} = \nabla_{df}(s),
\end{equation}
for all $f,g \in C^{\infty}(M)$ and $s \in \Gamma(E)$. The Jacobi identity follows readily from the flatness of $\nabla$.
\end{proof}

It is also useful to have an explicit description of the bivector associated to the Poisson bracket in \eqref{eq:expbracket}. To describe this one needs an auxiliary ordinary connection $\nabla$ on $E$. Then $\bar{\nabla}_{\alpha}(s) = \nabla_{\pi^{\sharp}(\alpha)}(s)$ defines a contravariant connection and the difference $\bar{\nabla} -\nabla^{\pi} \in \mathfrak{X}^1(M;\End(E))$ defines a vector field $X$. One now has that
\begin{equation}\label{eq:expbiv}
\pi_{E^*} = \text{hor}^{\nabla}(X) \wedge \mathcal{E} + \text{hor}^{\nabla}(\pi),
\end{equation}
is the bivector associated to $\set{\cdot,\cdot}_{E^*}$.

Using the $\rr^*$-invariance of the Poisson structure associated to a flat contravariant connection we can immediately conclude the following:
\begin{lemma}\label{lem:inducedcirclepoisson}
Let $E \rightarrow (M,\pi)$ be a line bundle endowed with a flat contravariant connection $\nabla$. Then the Poisson structure $\pi_{E^*}\in \mathfrak{X}^2(E^*)$ associated to $\nabla$ descends to a Poisson structure on $M \times S^1$.
\end{lemma}
\begin{proof}
Because the Poisson structure on $E^*$ is $\rr^*$-invariant, one can quotient by a $\mathbb{Z}_2 \times \mathbb{Z}$-action to obtain a Poisson structure on $M\times S^1$
\end{proof}
Given any Poisson manifold, we have an associated contravariant connection:
\begin{example}[The modular contravariant connection]
Let $(M^n,\pi)$ be a any Poisson manifold, and let $K : = \wedge^nTM$. The dual of this line bundle inherits a flat contravariant connection
\begin{align*}
\nabla^{\pi} : \Gamma(T^*M) \times \Gamma(K^*) &\rightarrow \Gamma(K^*),\\
(\alpha,s) &\mapsto \mathcal{L}_{\pi^{\sharp}(\alpha)}(s) - \inp{\pi,d\alpha}s,
\end{align*} 
called the \textbf{modular representation}. Consequently, by Proposition \ref{prop:poli} the total space of $K$ inherits an $\rr^*$-invariant Poisson structure, and therefore also a Poisson structure on the induces $S^1$-bundle.
\end{example}
This Poisson structure on $K$ is called the \textbf{modular Poisson structure} (see \cite{GP12}).

In the particular case when $(M,\pi)$ is a b-symplectic manifold we arrive at the regularisation:
\begin{lemma}
Let $(M,\pi)$ be a b-symplectic manifold. The Poisson structure on $M \times S^1$ induced by applying Lemma \ref{lem:inducedcirclepoisson} to the canonical contraviarant connection on $K^*$ and the one induced by the compact regularisation (i.e. Proposition \ref{prop:boris}) coincide. 
\end{lemma}
\begin{proof}
This can be easily verified using the explicit formula for the bivector on $K$ as in Equation \ref{eq:expbiv}.
\end{proof}

%------------------------------------------------------------
\subsection{Regularisation from a Jacobi perspective}\label{sec:Jacobi}

We study how the regularisation and the symplectisation of a b-contact manifold interact. The first thing to remark is that they commute:
\begin{proposition}
Let $(M,Z,\alpha)$ be a b-contact manifold. Then the symplectization of the contact-foliation $(M \times \rr,\ff_{\reg},\alpha_{\reg})$ coincides with the regularisation of the b-symplectic structure $(M\times \rr, d(e^t\alpha))$.
\end{proposition}
One can proof this by a direct check, or apply the more general Proposition \ref{prop:Poisandregcommute}.

\subsection{Jacobi structures}

Just as b-symplectic structures can be viewed as mildly degenerate Poisson structures, b-contact structures can be viewed as mildly degenerate Jacobi structures. In this appendix we will recall the relevant notions of Jacobi structures and show that the regularisation of a b-contact structure can be described in a Jacobi framework as well. This framework is a direct analogue of the Poisson case, but to the best of our knowledge it has not been discussed in the literature before. Nevertheless, there are some interesting differences between the Poisson and Jacobi case.
\begin{definition}
A \textbf{Jacobi structure} on a manifold is a line bundle $L$ endowed with a Lie bracket on its sections $\{\cdot,\cdot \}$ with the property that $D_e :=\{e, \cdot\} :\Gamma(L) \rightarrow \Gamma(L)$ is a first order differential operator for all $e \in \Gamma(L)$.
\end{definition}

\begin{definition}
A \textbf{Jacobi pair} $(\Lambda,R) \in \mathfrak{X}^2(M) \times \mathfrak{X}^1(M)$ is a pair satisfying the following relations:
\begin{equation*}
[\Lambda,\Lambda] = \Lambda \wedge R, \quad [\Lambda,R] = 0. \hfill \qedhere
\end{equation*}
\end{definition}

\begin{proposition}
There is a one-to-one correspondence between Jacobi structures $\pbre$ on the trivial line bundle, and Jacobi pairs $(\Lambda,R)$ given by:
 \begin{equation*}
\{f,g\} = \Lambda(df,dg) + fR(g) - gR(f), \quad \text{for all } f,g \in \Gamma(\underline{\rr}) = C^{\infty}(M).
\end{equation*}
\end{proposition}

\subsubsection{Poissonization}
One way of dealing with Jacobi structures, is to consider an associated Poisson structure. We will first recall this construction for Jacobi pairs:
\begin{definition}
Let $(M,\Lambda,R)$ be a Jacobi manifold. Then
\begin{equation*}
\tilde{\Lambda} := e^{-s}\left( \Lambda +  \partial_{s} \wedge R\right),
\end{equation*}
defines a Poisson structure on $M \times \rr$ called the \textbf{Poissonization}.
\end{definition}
However, the Poissonization in this form does not extend to general Jacobi structures on line bundles, first a coordinate transformation $s \mapsto \pm\log(s)$ is necessary, resulting in two possible Poisson structures on $M \times \rr^*$:
\begin{equation}\label{eq:Poissoni}
\Pi_1 = t(\Lambda - t\partial_t \wedge R),\quad \Pi_2 = t^{-1}(\Lambda + t\partial_t \wedge R).
\end{equation}
We will now recall how to describe the Poissonization of a general Jacobi structure:
\begin{definition}\label{def:Poisson1}
Let $\pbre$ be a Jacobi bracket on $p : L \rightarrow M$. Let $s_1,s_2 \in \Gamma(L)$ and let $F_1,F_2 \in C^{\infty}(L\backslash \set{0})$ be the unique functions satisfying $s_i(p(e)) = F_i(e)e$, for all $e \in L\backslash \set{0}$. Then the Poisson bracket on $L\backslash \set{0}$ defined by
\begin{equation}
\pbr{F_1,F_2}_{\Pi_1}(e)e = \pbr{s_1,s_2}(p(e)),
\end{equation} 
is called the \textbf{Poissonization} of $\pbre$.
\end{definition}
When $L$ is a trivial bundle, the Poisson structure defined in Definition \ref{def:Poisson1} coincides with $\tilde{\Lambda}_1$ from Equation \ref{eq:Poissoni}. If $s(x) = (x,f(x))$ is a section of the trivial line bundle, then the corresponding function on $L\backslash \set{0}$ is given by $F(x,t) = t^{-1}f(x)$. Therefore, one may write the defining equation of the Poissonization in the elegant homogeneous form
\begin{equation*}
\pbr{f_1t^{-1},f_2t^{-1}}_{\Pi_1}t = \pbr{f_1,f_2}. 
\end{equation*}
\begin{definition}\label{def:Poisson2}
Let $\pbre$ be a Jacobi bracket on $p : L \rightarrow M$. Let $s_1,s_2 \in \Gamma(L)$, and let $\hat{s}_1,\hat{s}_2 \in C^{\infty}(L^*)$ denote the corresponding homogeneous functions on $L^*$. Then $\Pi_2 \in \mathfrak{X}^2(L^*\backslash \set{0})$, defined by
\begin{equation*}
\Pi_2(d\hat{s}_1,d\hat{s}_2) = \widehat{\Delta(s_1,s_2)},
\end{equation*}
is called the \textbf{Poissonization} of $\pbre$.
\end{definition}
When $L$ is a trivial bundle, the Poisson structure defined in Definition \ref{def:Poisson1} coincides with $\tilde{\Lambda}_2$ from Equation \ref{eq:Poissoni}.  If $s(x) = (x,f(x))$ is a section of the trivial line bundle, then the corresponding homogeneous function is given by $\hat{s}(x,t) = f(x)t$. Therefore, one may write the defining equation of the Poissonization in the elegant form
\begin{equation*}
\pbr{f_1t,f_2t}_{\Pi_2} = t\pbr{f_1,f_2}.
\end{equation*}

Recall that the pairing $L\backslash \set{0} \times L^*\backslash \set{0} \rightarrow \rr$ induces a diffeomorphism $L\backslash \set{0} \simeq L^*\backslash \set{0}$. When $L$ is a trivial line bundle, this is simply the map $(x,t) \mapsto (x,t^{-1})$. Consequently, it is no surprise that this diffeomorphism sends the Poisson structure defined in Definition \ref{def:Poisson1} to the one from Definition \ref{def:Poisson2}. Although, this motivates that both may be called the Poissonization, admittedly it can be confusing at times which one is implied. From here on onward, when we talk about the Poissonization we will imply the one on $L^*\backslash \set{0}$.
\begin{remark}
Note that the Poissonization on $L\backslash \set{0}$ can in fact be extended over the zero-section (by zero), whereas the Poissonization on $L^*\backslash \set{0}$ cannot.
\end{remark}

%\begin{definition}
%Let $e \in \Gamma(L)$. The \textbf{Hamiltonian vector field} of a Jacobi structure $\pbre$ on $L$ is defined to be the symbol of $\pbr{e,\cdot}$.
%\end{definition}
%When $(\Lambda,R)$ is a Jacobi pair, its Hamiltonian vector field is given by $\Lambda^{\sharp}(df) + fR$.

%\begin{remark}
%The fact that the log symplectic manifold in \ref{prop:reglogsymp} need not be orientable does not provide more examples of symplectic foliations. Indeed if $(M,\omega)$ is non-orientable log symplectic, then the connected double cover is also log symplectic and doing the regularisation for both results in the same manifold.
%\end{remark}

Recall that a Jacobi structure induces a Lie algebroid structure on its jet bundle:
\begin{definition}
Let $(L,\{\cdot,\cdot\})$ be a Jacobi structure. Then $J^1L$ inherits the structure of a Lie algebroid, defined uniquely by the relations:
\begin{align*}
[j^1e_1,j^1e_2] = j^1\{e_1,e_2\}, \quad \rho(j^1e_1) = \sigma_{D_{e_1}},
\end{align*}
for all $e_1,e_2 \in \Gamma(L)$.
\end{definition}
It is instructive to spell out the Lie algebroid structure explicitly for the case of a trivial line bundle, in which $J^1\underline{\rr} \simeq T^*M \oplus \underline{\rr}$:
\begin{definition}
The Lie algebroid of the Jacobi pair $(\Lambda,R)$ on $M$ is the vector bundle $T^*M \oplus \rr$, with the anchor $\rho : T^*M \oplus \rr \rightarrow TM$ given by
\begin{equation*}
\rho(\alpha,\lambda) = \Lambda^{\sharp}(\alpha) + \lambda R,
\end{equation*}
and the bracket uniquely defined by
\begin{align*}
[(\alpha,0),(\beta,0)] &= ([\alpha,\beta]_{\Lambda} - \iota_R(\alpha \wedge  \beta),\Lambda(\alpha,\beta))\\
[(0,1),(\alpha,0)] &= (L_R(\omega),0). \hfill \qedhere
\end{align*}
\end{definition}

\begin{theorem}
Let $(L,\pbre)$ be a Jacobi structure. Then the the even dimensional leaves of $\rho(J^1L)$ inherit a locally conformal symplectic structure whereas the odd dimensional leaves inherit a contact structure.
\end{theorem}
\subsection{Induced homogeneous Jacobi structures}

Let $(L \rightarrow M,\pbre_1)$ be a Jacobi structure, $p_E : E \rightarrow M$ a vector bundle and $\nabla$ a flat $J^1L$-connection on $E^*$. We can produce a Jacobi structure $(p^*_EL\rightarrow E,\pbre_2)$ as follows:
\begin{align*}
\pbr{p^*e_1,p^*e_2}_2 = p^*\pbr{e_1,e_2},\quad \text{for all } e_1,e_2 \in \Gamma(L),\\
\sigma_{J_2}(p^*e)(f) = \nabla_{j^1e}(f) \quad \text{for all }f \in \Gamma(E^*) \subset C^{\infty}(E).
\end{align*}
This in fact defines a one-to-one correspondence, mimicking Proposition \ref{prop:poli}:
\begin{proposition}\label{prop:jaccor}
Let $E \rightarrow M$ be a line bundle and $(L,\pbre)$ a Jacobi structure. There is a one-to-one correspondence between:
\begin{itemize}
\item $\rr^*$-invariant Jacobi structures $(p^*_EL,\pbre_2)$ on $E$, i.e., satisfying $m_{\lambda}(\pbr{s_1,s_2}) = \pbr{m_{\lambda}(s_1),m_{\lambda}(s_2)}$ for all $s_1,s_2 \in \Gamma(p_E^*L)$.
%, i.e. $m_{\lambda}^*\Lambda = \Lambda$ and $m_{\lambda}^*R = R$
\item Flat $(J^1L)$-connections $\nabla^J$ on $E^*$.
\end{itemize}
\end{proposition}

When the Jacobi bracket on $L$ is induced by a Jacobi pair $(\Lambda,R)$ we can give a global description of $\pbre_2$. Let $\nabla$ be a flat ordinary connection on $E^*$, and define a $(T^*M\oplus \rr)$-connection on $E^*$ via $\bar{\nabla}_{(\alpha,f)} = \nabla_{\pis{\alpha}+fR}$. Then using the difference 
\begin{equation}
\nabla^J - \bar{\nabla} = (X,f) \in \mathfrak{X}^1(M)\times C^{\infty}(M),
\end{equation}
one finds that the Jacobi pair
\begin{equation}
(\tilde{\Lambda},\tilde{R}) := (\hor^{\nabla}(X) \wedge \mathcal{E} + \hor^{\nabla}(\Lambda),\hor^{\nabla}(R) + f\mathcal{E}),
\end{equation}
corresponds with $\pbre_2$ as defined above. 
\subsubsection{Canonical connections}
Different from the Poisson case, associated to a Jacobi structure $(L,\pbre)$ there are several canonical $J^1L$-representations:
\begin{align*}
\nabla^1 : \Gamma(J^1L) \times \Gamma(L) &\rightarrow \Gamma(L)\\
(j^1s,e) &\mapsto \pbr{s,e}\\
\nabla^2 : \Gamma(J^1L) \times \Gamma(\det(J^1L)) &\rightarrow \Gamma(\det(J^1L))\\
(j^1s,\omega) &\mapsto L_{D_s}(\omega)\\
\nabla^3 : \Gamma(J^1L) \times \Gamma(\det(T^*M)) &\rightarrow \Gamma(\det(T^*M))\\
(j^1s,\Omega) &\mapsto L_{\rho(j^1s)}(\Omega)
\end{align*}
Here $L_D : \Gamma(\det(J^1L)) \rightarrow \Gamma(\det(J^1L))$ is induced by the Lie derivative on $J^1L$ defined via:
\begin{equation*}
\inp{L_D(\theta),D'} := D(\inp{\theta,D'}) - \inp{\theta,[D,D']}, \quad D,D' \in \Gamma(DL), \theta \in \Gamma(J^1L),
\end{equation*}
where $\inp{\cdot,\cdot} : J^1L\otimes DL \rightarrow L$ is the natural pairing.

Again, it is instructive to see these representations in the case when $L = \underline{\mathbb{R}}$, and the Jacobi structure corresponds to a Jacobi pair $(\Lambda,R)$:
\begin{align*}
\nabla^1_{(df,f)}(\underline{1}) &= -R(f),\\
\nabla^2_{(df,f)}(\omega) &= L_{\Lambda^{\sharp}(df)+fR}\omega - R(f)\omega, \quad \omega \in \det(J^1\underline{\rr}) \simeq \det(T^*M)\\
\nabla^3_{(df,f)}(\Omega) &= L_{\Lambda^{\sharp}(df)+fR}\Omega
\end{align*}

%From these formula it is already apparent that there is a relation between these representations:
%\begin{lemma}
%Let $(L,\pbre)$ be a Jacobi structure. Then $\det(J^1L)\otimes L \simeq \det(T^*M)$, and under this isomorphism $\nabla^1\otimes \nabla^2$ corresponds to $\nabla^3$.
%\end{lemma}
%\begin{proof}
%From the formulae for the case $L = \underline{\rr}$ the result is immediate provided we can obtain a global isomorphism $\det(J^1L)\otimes L \simeq \det(T^*M)$. First, recall that $J^1L \otimes L^*$ is canonically isomorphic to $(DL)^*$, therefore it suffices to find an isomorphism $\det(DL) \simeq \det(TM)$ This isn't true, the isomorphism doesn't pass through the determinant (silly me). Recall that any ordinary connection $\nabla$ on $L$ induces an isomorphism
%\begin{align*}
%\Gamma(DL) &\rightarrow \mathfrak{X}(M) \oplus \rr\cdot\mathcal{E},\\
%D &\mapsto (\sigma_D,D-\nabla_{\sigma_D}),
%\end{align*}
%although this clearly depends on the connection, the induced isomorphism %$\det(DL) \simeq \det(TM)\otimes \underline{\rr}$ does not.
%\end{proof}
The Lie algebroid structure on $J^1L$ induced by a Jacobi structure, also inherits the canonical representation (in the sense of Definition \ref{def:canonicalRep})
\begin{equation*}
\nabla^4 : \Gamma(J^1L) \times \Gamma(\det(J^1L)\otimes \det(T^*M)) \rightarrow \Gamma(\det(J^1L)).
\end{equation*}
\begin{proposition}
Let $(L,\pbre)$ be a Jacobi structure. Then $\nabla^4 = \nabla^3 \otimes \nabla^2$.
\end{proposition}
\begin{remark}
Note the similarities with the Poisson case. There, the canonical $T^*M$-representation is defined on $(\det(T^*M)^{\otimes 2})$, and the modular representation provided a square-root for this representation.
\end{remark}
Having established that any $J^1L$-representation induces a Jacobi structure, and that any Jacobi structure induces three canonical $J^1L$-representations we will go on and study these.

Although the $J^1L$-representation on $L$ is rather tautological, the induced Jacobi structure is already quite interesting:
\begin{example}[Conformal Poissonisation]
Let $(\Lambda,R)$ be a Jacobi pair. The Jacobi pair $(\tilde{\Lambda},\tilde{R})$ induced by $\nabla^1$ is given by
\begin{equation*}
(\tilde{\Lambda},\tilde{R}) = (\Lambda + R \wedge t\partial_t,R).
\end{equation*}
Note that the Poissonification of the Jacobi pair, is conformally equivalent to this one with conformal factor $t^{-1}$.

Now let $(L,\pbre)$ be a general Jacobi structure, and consider the Jacobi structure $(p^*_{L^*}L,\pbre_2)$ induced by $\nabla^1$. The Poissonization of $(L,\pbre)$ is a Poisson structure on $L^*\backslash \set{0}$, we will now again compare the two.

Recall that $p_{L^*\setminus \set{0}}^*(L) \simeq (L^*\backslash \set{0}) \times \rr$. Consequently the restriction of $\pbre_2$ to $q^*(L)$ is induced by a Jacobi pair on $L^*\backslash \set{0}$. We would like to say that this Jacobi pair is conformally equivalent to the Poissonization, but we don't have a global conformal factor. This is because there is no global analogue of the function $t^{-1}$ on $L^*\backslash \set{0}$. Still, in each trivialisation of $L^*$ these Poisson structures are conformally equivalent with conformal factor $t^{-1}$.
\end{example}

We would now like to define the equivalent of the modular Poisson structure for Jacobi structures. However, there seem to be two natural choices, namely $\nabla^2$ and $\nabla^3$. Nevertheless, for reasons which will become clear momentarily we define:
\begin{definition}
Let $(L,\pbre)$ be a Jacobi structure. We call $\nabla^2 : \Gamma(J^1L) \times \Gamma(\det(J^1L)) \rightarrow \Gamma(\det(J^1L))$ the \textbf{modular representation} of $(L,\pbre)$. And the induced Jacobi structure on $\det(J^1L)$ the \textbf{modular Jacobi structure}.
\end{definition}
We want to describe how this modular Jacobi structure interacts with the Poissonization, for which we need the following Lemma:
\begin{lemma}
Let $L$ be a real line bundle. Then
\begin{equation}
(p^*_{\det(J^1L)^*}L)^*\backslash \set{0} \simeq \det(T^*(L^*\backslash \set{0}).
\end{equation}
\end{lemma}
\begin{proof}
This follows from a couple of straightforward observations for arbitrary vector bundles $E,F \rightarrow M$:
\begin{itemize}
\item $p_E^*(F)\backslash \set{0} \simeq p^*_{F\backslash \set{0}}(E)$
\item For any $f : N\rightarrow M$, $f^*\det(E) \simeq \det(f^*E)$,
\end{itemize}
combined with the fact that $T(L^*\backslash \set{0}) \simeq p_{L^*\backslash \set{0}}^*(J^1L)$.
\end{proof}

\begin{proposition}\label{prop:Poisandregcommute}
Let $(L,\pbre)$ be a Jacobi manifold. Then Poissonization and taking the modular Poisson/Jacobi structure commute. That is, the following diagram commutes:
\begin{center}
\begin{tikzcd}[row sep = large, column sep = huge]
(p^*_{\det(J^1L)^*}L,\pbre_{\rm mod}) \ar[r,"{\rm Poissonization}"] & (\det(T^*(L\backslash \set{0})),\pi_{\rm mod})\\
(L,\pbre) \ar[r,"{\rm Poissonization}"] \ar[u,"{\rm Modular Jacobi}"] & (L^*\backslash \set{0},\Pi_2) \ar[u,"{\rm Modular Poisson}"]
\end{tikzcd}
\end{center} 
\end{proposition}
\begin{proof}
It suffices to show this for Jacobi pairs, for which it is a straightforward computation.
\end{proof}

To end this section we recall that b-contact forms can be viewed as Jacobi structures:
\begin{proposition}[\cite{MO18}]
Let $(M,Z)$ be a b-pair. There is a one-to-one correspondence between:
\begin{itemize}
\item b-contact forms $\alpha \in \Omega^1(\A_Z)$.
\item Jacobi structures $(\Lambda,R)$ for which $\Lambda^n \wedge R$ vanishes transversely along $Z$.
\end{itemize}
\end{proposition}

Having established the appropriate analogue of the modular Jacobi structure, we can now relate it with the regularisation procedure:
\begin{proposition}
Let $(M,Z,\alpha)$ be a b-contact form, and let $(\Lambda,R)$ denote the associated Jacobi pair. The contact foliation given by the intrinsic regularisation of $\alpha$ coincides with the contact foliation on $M \times \rr^*$ underlying the modular Jacobi structure of $(\Lambda,R)$.
\end{proposition}

\section{Recap: Convex surface theory}\label{sec:convex}

Convex surfaces play an essential role in the study of contact manifolds, and their theory was developed by Giroux in dimension 3 in \cite{Gir91}, and in \cite{HH19} in higher dimensions. Also in the study of $b$-contact structures they turn out to play an important role.
\begin{definition}
Let $(M,\xi = \ker \alpha)$ be a contact manifold. A \textbf{contact vector field} $X \in \mathfrak{X}^1(M)$ is a vector field for which $(\mathcal{L}_X\alpha)|_{\xi} = 0$. For a function $f \in C^{\infty}(M)$ the unique contact vector field satisfying $\alpha(X_f) = f$ is the \textbf{Hamiltonian vector field} associated to $f$.
\end{definition}
\begin{definition}
A hypersurface $\Sigma \subset (M,\xi)$ is \textbf{convex} if there exists a contact vector field transverse to $\Sigma$. 
\end{definition}
Note that when $\Sigma \subset (M,\xi)$ is convex it is in particular co-orientable. For any co-orientable hypersurface $\Sigma \subset (M,\xi)$ we may write $\alpha = u_tdt + \beta_t$ with $u_t \in C^{\infty}(M)$ and $\beta_t \in \Omega^1(M)$. Convexity is equivalent with the fact that $u_t$ and $\beta_t$ are pull-backs from $Z$.

Convex hypersurfaces enter the study of b-contact forms via the following result:
\begin{theorem}[\cite{MO18}]\label{th:surgery}
Let $(M,\xi)$ be a contact manifold and let $Z$ be a convex hypersurface in $M$. Then $M$ admits a $b^{2k}$-contact structure for all $k$ that has $Z$ as critical set. Also $M$ admits a $b^{2k+1}$-contact structure for all $k$ that has two diffeomorphic connected components $Z_1$ and $Z_2$ as critical
set, one of which can be taken to coincide with $Z$.
\end{theorem}

\begin{definition}
Let $\Sigma \subset (M,\xi =\ker \alpha)$ be a convex hypersurface with transverse contact vector field $X$. The \textbf{dividing set} is defined as $\Gamma(\Sigma) = \set{x\in \Sigma : \alpha_x(X_x) = 0}$ and its complement $R_{\pm}(\Sigma) = \set{\pm \alpha_x(X_x) \geq 0}$.
\end{definition}
\begin{remark}
These definitions are consistent: The dividing set of a b-contact structure obtained using Theorem \ref{th:surgery} coincides with the one of the convex hypersurface.
\end{remark}

In dimension three $\Gamma(\Sigma)$ is a disjoint collection of curves, which determines the contact structure (up to isotopy) completely around $\Sigma$. Moreover, overtwistedness of a neighbourhood of $\Sigma$ can be easily read of from the dividing set:
\begin{lemma}[Giroux' criterion]\label{lem:Giroux}
Let $(M^3,\xi)$ be a contact manifold, and let $\Sigma$ be a convex hypersurface. Then $\Sigma$ has a tight neighbourhood if and only if either:
\begin{itemize}
\item no component of $\Gamma(\Sigma)$ is contractible in $\Sigma$,
\item $\Sigma$ is a sphere and $\Gamma(\Sigma)$ is connected.
\end{itemize}
\end{lemma}

In higher dimensions the situation is more complicated, and $\Gamma$ is a codimension two submanifold of $M$. Moreover, the dividing set will inherit the structure of an ideal Liouville domain:
\begin{definition}[\cite{Gir17}]\label{def:idealdomain}
An \textbf{ideal Liouville domain} is a manifold with boundary $F$, endowed with an exact symplectic form $d\lambda$ on ${\rm Int }F$ such that there exists a defining function $u$ for $\partial F$ such that $u \lambda$ extends to, and induces a contact form on, $\partial F$.
\end{definition}
\begin{lemma}[\cite{HH19}]\label{lem:hondaconvex}
Let $\Sigma \subset (M,\xi)$ be a convex hypersurface. Then
\begin{itemize}
\item $\Gamma(\Sigma)$ is a codimension-two contact submanifold and
\item $R_{\pm}$ inherits the structure of an ideal Liouville domain, for which the induced contact structure on the boundary coincides with $\rest{\xi}{\Gamma}$.
\end{itemize}
\end{lemma}
We include the proof for completeness, and future reference:
\begin{proof}
Convexity ensures that there exists a neighbourhood of $\Sigma$ such that $\alpha = udt + \beta$, with $u \in C^{\infty}(\Sigma)$ and $\beta \in \Omega^1(\Sigma)$. The contact condition states that
\begin{align*}
\alpha \wedge (d\alpha)^n &= udt \wedge (d\beta)^n + \beta \wedge (d\beta)^n + n\beta \wedge du\wedge dt \wedge (d\beta)^{n-1}\\
&= udt \wedge (d\beta)^n + n\beta \wedge du\wedge dt \wedge (d\beta)^{n-1},\\
&= dt \wedge (u(d\beta)^n + n\beta \wedge du \wedge (d\beta)^{n-1})
\end{align*}
where the last term dropped out because of dimension reasons. Now define $\lambda = \beta/u$, then $d\lambda = u^{-1}d\beta + u^{-2}\beta \wedge du$, and consequently
\begin{align*}
(d\lambda)^n &= u^{-n}(d\beta)^n + n u^{-n-1}(d\beta)^{n-1}\beta \wedge du\\
&= \frac{1}{u^{n+1}}(u(d\beta)^n + n (d\beta)^{n-1} \wedge \beta \wedge du) ,
\end{align*}
and is therefore symplectic. We see that $R_{\pm}$ has the structure of ideal Liouville domain, with the contact structure on the boundary induced by $\beta$.
\end{proof}
%\begin{remark}
%The reason for the particular rescaling in the previous proof is that any higher power (might) not work because 
%\begin{equation}
%d(\beta/u^k)^n = u^{-n-k}(u(d\beta)^n + k n (d\beta)^{n-1} \wedge \beta \wedge du),
%\end{equation}
%and the $k$ in the second term screws things up. Of course, near $\set{u=0}$ the above is still symplectic but further away we can't guarantee this (\alert{although it is a bit silly}).
%
%Note that for $k$ even, $d(\beta/u^k)$ defines a scattering symplectic form around $\set{u=0}$ (\alert{therefore it would be nice to extend it.})
%\end{remark}

The fact that the contact manifold $\Gamma(\Sigma)$ can be realized as the boundary of a Liouville manifold immediately implies the following:
\begin{lemma}\label{lem:tight}
Let $\Sigma \subset (M,\xi)$ be a convex hypersurface. Then $(\Gamma(\Sigma),\rest{\xi}{\Gamma(\Sigma)})$ is a tight contact manifold.
\end{lemma}
\begin{proof}
Again, consider the decomposition $\alpha = u dt + \beta$ as before, and define $\Gamma_{\varepsilon} = u^{-1}(\varepsilon)$. Note, that for $\varepsilon$ small $\Gamma_{\varepsilon}$ is diffeomorphic to $\Gamma$. Moreover, $\Gamma_{\varepsilon}$ is (one component of) the boundary of $u^{-1}([\varepsilon,\infty))$ and therefore fillable. But as the contact structure on $\Gamma_{\varepsilon}$ is simply $\beta/\varepsilon$, we conclude that it is contactomorphic to $\Gamma$, which is then fillable.
\end{proof}

\bibliographystyle{alpha}
\bibliography{references} 
%\begin{thebibliography}{xxxx}
%\bibitem{P16} P.\ Pansu. \emph{Submanifolds and differential forms on Carnot manifolds, after M.\ Gromov and M.\ Rumin}.
%\end{thebibliography}
\end{document}